\documentclass[11pt,reqno]{amsart}

\usepackage[utf8]{inputenc}
\usepackage[T1]{fontenc}

\usepackage{amssymb}
\usepackage{amsmath}
\usepackage{amsfonts}
\usepackage{latexsym}
\usepackage{amsthm}
\usepackage{mathrsfs}
\usepackage[colorlinks, urlcolor=blue, citecolor=black, linkcolor=black ,breaklinks]{hyperref} 

\usepackage{enumerate}

\newtheorem{Definition}{Definition}[section]
\newtheorem{Theorem}[Definition]{Theorem}
\newtheorem{Lemma}[Definition]{Lemma}
\newtheorem{Corollary}[Definition]{Corollary}
\newtheorem{Proposition}[Definition]{Proposition}

\def\R{\mathbb{R}}

\def\N{\mathbb{N}}
\def\F{\mathbb{F}}

\newcommand{\Ld}[2][D]{\ensuremath{\mathcal{L}(#2\mathcal{#1})}} 
\newcommand{\D}[1][D]{\ensuremath{\mathcal{#1}}} 

\newcommand{\mus}{\ensuremath{\mu^\mathrm{sym}}}

\newcommand{\cA}{\mathcal{A}}

\begin{document}

\title[New uniform and asymptotic upper bounds on the tensor rank]{New uniform and asymptotic upper bounds on the tensor rank of multiplication in extensions of finite fields}
\author{Julia Pieltant}
\address{Inria Saclay,
LIX, \'Ecole Polytechnique,
91128 Palaiseau Cedex, France}
\email{pieltant@lix.polytechnique.fr}
\author{Hugues Randriam}
\address{ENST (``Telecom ParisTech'')\\ 46 rue Barrault, F-75634 Paris Cedex 13\\ France}
\email{randriam@telecom-paristech.fr}
\date{\today}
\keywords{Algebraic function field, tower of function fields, tensor rank, algorithm, finite field}
\subjclass[2000]{Primary 14H05; Secondaries 11Y16, 12E20}

\begin{abstract}
We obtain new uniform upper bounds for the (non necessarily symmetric) tensor rank of the multiplication in the extensions of the finite fields $\F_q$ for any prime or prime power $q\geq2$; moreover these uniform bounds lead to new asymptotic bounds as well. In addition, we also give purely asymptotic bounds which are substantially better by using a family of Shimura curves defined over $\F_q$, with an optimal ratio of $\F_{q^t}$-rational places to their genus where $q^t$ is a square.
\end{abstract}

\maketitle


\section{Introduction}

	\subsection{Tensor rank of multiplication}

Let $K$ be a field and let $\cA$ be a finite-dimensional $K$-algebra. We denote by $m_{\cA}$ the multiplication map of $\cA$. It can be seen as a $K$-bilinear map from ${\cA \times \cA}$ into $\cA$, or equivalently, as a linear map from the tensor product ${\cA \bigotimes \cA}$ over $K$ into $\cA$. One can also represent it by a tensor ${t_{\cA} \in \cA^\star\bigotimes \cA^\star \bigotimes \cA}$ where $\cA^\star$ denotes the dual of $\cA$ over $K$. Hence the product of two elements $x$ and $y$ of $\cA$ is the convolution of this tensor with ${x \otimes y \in \cA \bigotimes \cA}$. If 
\begin{eqnarray}\label{eq_tensor}
t_{\cA} = \sum_{l=1}^{\lambda} a_l \otimes b_l \otimes c_l
\end{eqnarray}
where ${a_l \in  \cA^\star}$, ${b_l \in  \cA^\star}$, ${c_l \in  \cA}$, then 
\begin{eqnarray}\label{eq_algo}
x\cdot y=\sum_{l=1}^{\lambda} a_l(x)b_l(y)c_l.
\end{eqnarray}

Every expression (\ref{eq_algo}) is called a bilinear multiplication algorithm $\mathcal U$ for $\cA$ over $K$. The integer $\lambda$ is called the bilinear complexity  ${\mu({\mathcal U})}$ of $\mathcal U$.\\
Let us set
$$
\mu_K(\cA)= \min_\mathcal{U} \mu(\mathcal{U}),
$$
where $\mathcal U$ is running over all bilinear multiplication algorithms for $\cA$ over $K$.\\
Then $\mu_K(\cA)$ corresponds to the minimum possible number of summands in any tensor decomposition of type (\ref{eq_tensor}), which is the rank of the tensor of multiplication in $\cA$ over $K$. The tensor rank $\mu_K(\cA)$ is also called the bilinear complexity of multiplication in $\cA$ over $K$. 

 When the decomposition (\ref{eq_tensor}) is symmetric, i.e. ${a_l=b_l}$ for all ${l=1,\ldots,\lambda}$, we say that the corresponding algorithm $\mathcal U$ is a symmetric bilinear multiplication algorithm. If we focus on such algorithms, then the corresponding complexity is called the symmetric bilinear complexity of multiplication in $\cA$ over $K$ and we set:
$$
\mus_K(\cA) = \min_{\mathcal{U}^\mathrm{sym}} \mu(\mathcal{U}^\mathrm{sym}),
$$
with $\mathcal{U}^\mathrm{sym}$ running over all symmetric bilinear multiplication algorithms for $\cA$ over $K$. Note that one has
$$
\mu_K(\cA) \leq \mus_K(\cA).
$$
In this work we will be mainly interested in the case where $K=\F_q$ is the finite field with $q$ elements (where $q$ is a prime power) and $\cA=\F_{q^n}$ is the extension field of degree $n$ of $\F_q$. We then set
$$
\mu_q(n)=\mu_{\F_q}(\F_{q^n}).
$$
However for technical reasons we will also need the quantities
$$
\mu_q(m,l)=\mu_{\F_q}(\F_{q^m}[t]/(t^l))
$$
so that $\mu_q(n)=\mu_q(n,1)$.\\
Similarly, we set  ${\mus_q(n)=\mus_{\F_q}(\F_{q^n})}$ and ${\mus_q(m,l)=\mus_{\F_q}(\F_{q^m}[t]/(t^l))}$.

	\subsection{Notations} 

Let $F/\F_q$ be an algebraic function field of one variable of genus $g$, with constant field $\F_q$, associated to a curve $X$ defined over $\F_q$. 
For any place $P$ we define $F_P$ to be the residue class field of $P$ and $\mathcal{O}_P$ its valuation ring. Every element ${t \in P}$ such that ${P = t \mathcal{O}_P}$ is called a local parameter for $P$ and we denote by ${v_P}$ a discrete valuation associated to the place $P$ of $F/\F_q$. Recall that this valuation does not depend on the choice of the local parameter. Let ${f \in F\backslash \{0\}}$, we denote by ${(f) := \sum_P v_P(f) P}$ where $P$ is running over all places in $F/\F_q$, the principal divisor of $f$. If $\D$ is a divisor then ${\Ld{}=\{f \in F/\F_q ; \D + (f) \geq 0\}\cup \{0\}}$ is a vector space over $\F_q$ whose dimension $\dim {\D}$ is given by the 
Riemann-Roch Theorem. The degree of a divisor ${\D}=\sum_P a_P P$ is defined by ${\deg \D=\sum_P a_P \deg P}$ where $\deg P$ is the dimension of $F_P$ over $\F_q$. The order of a divisor ${\D=\sum_P a_P P}$ at $P$ is the integer $a_P$ denoted by ${\mathrm{ord}_P \, \D}$. The support of a divisor $\D$ is the set ${\mathrm{supp} \, \D}$ of the places $P$ such that ${\mathrm{ord}_P \, \D \neq 0}$.  Two divisors $\D$ and $\D'$ are said to be equivalent if ${\D=\D'+(x)}$ for an element ${x \in F\backslash \{0\}}$.

We denote by $B_k(F/\F_q)$ the number of places of degree $k$ of $F$ and by $g(F/\F_q)$ the genus of $F/\F_q$.

	\subsection{Known results}\label{knownresults}

The bilinear complexity $\mu_q(n)$ of the multiplication in the $n$-degree extension of a finite field $\F_q$ is known for certain values of $n$.  In particular, S. Winograd \cite{wino3} and H. de Groote \cite{groo} have shown that this complexity is ${\geq 2n-1}$, with equality holding if and only if ${n \leq \frac{1}{2}q+1}$. Moreover, in this case one has ${\mus_q(n)=\mu_q(n)}$.
Using the principle of the D.V. and G.V. Chudnovsky algorithm \cite{chch} applied to elliptic curves, M.A. Shokrollahi has shown in \cite{shok} that the symmetric bilinear complexity of multiplication is equal to $2n$ for ${\frac{1}{2}q +1< n < \frac{1}{2}(q+1+{\epsilon (q) })}$ where $\epsilon$ is the function defined by:
$$
\epsilon (q) = \left \{
	\begin{array}{l}
 		 \mbox{the greatest integer} \leq 2{\sqrt q} \mbox{ prime to $q$, if $q$ is not a perfect square} \\
  		2{\sqrt q}\mbox{, if $q$ is a perfect square.}
	\end{array} \right .
$$

Moreover, U. Baum and M.A. Shokrollahi have succeeded in \cite{bash} to construct effective optimal algorithms of type Chudnovsky in the elliptic case. 

Recently in \cite{ball1}, \cite{ball3}, \cite{baro1}, \cite{balbro}, \cite{balb}, \cite{bach} and \cite{ball5} the study made by M.A. Shokrollahi has been  generalized to algebraic function fields of genus~$g$. \\

Let us recall that the original algorithm of D.V. and G.V. Chudnovsky introduced in \cite{chch} leads to the following theorem:

\begin{Theorem}
Let $q=p^r$ be a power of the prime $p$. The symmetric tensor rank $\mus_q(n)$ of multiplication in any finite field $\F_{q^n}$ is linear with respect to the extension degree; more precisely, there exists a constant $C_q$ such that:
$$
\mus_q(n) \leq C_q n.
$$
\end{Theorem}

Moreover, one can give explicit values for $C_q$:

\begin{Proposition} \label{theobornes}
The best known values for the constant $C_q$ defined in the previous theorem are:
$$
C_q = \left \{
	\begin{array}{lll}
		\hbox{if } q=2 & \hbox{then } 22 &   \hbox{\cite{ceoz} and \cite{bapi}} \cr
		\hbox{else if } q=3 &  \mbox{then } 27 & \hbox{\cite{ball1}} \cr
		\hbox{else if } q=p \geq 5 &  \mbox{then }   3(1+ \frac{4}{q-3}) & \hbox{\cite{bach}} \\
		\hbox{else if } q=p^2 \geq 25 & \hbox{then }   2(1+\frac{2}{\sqrt{q}-3}) & \hbox{\cite{bach}} \\
		\hbox{else if } q=p^{2k} \geq 16 & \hbox{then } 2(1+\frac{p}{\sqrt{q}-3}) & \hbox{\cite{ball3}} \\
		\hbox{else if } q \geq 16 & \hbox{then }  3(1+\frac{2p}{q-3}) & \hbox{\cite{baro1}, \cite{balbro} and \cite{balb}}\\
		\hbox{else if } q > 3 & \hbox{then }  6(1+\frac{p}{q-3}) & \hbox{\cite{ball3}}.
	\end{array}\right .
$$
\end{Proposition}

In order to obtain these good estimates for the constant $C_q$, S. Ballet has given in \cite{ball1} some easy to verify conditions allowing the use of the D.V. and G.V. Chudnovsky  algorithm. Then S. Ballet and R. Rolland have generalized in \cite{baro1} the algorithm using places of degree one and~two.

Recently, various generalizations of this algorithm were introduced in \cite{rand3}.
We will use the version that can be found in \cite[Proposition 5.7]{rand3} and which, expressed in the language of function fields, reads as follows:

\begin{Theorem}\label{theobornegenerale}
Let $F/\F_q$ be an algebraic function field of genus $g\geq2$, and let $m,l\geq1$ be two integers.

Suppose that $F$ admits a place of degree $m$
(a sufficient condition for this is $2g+1\leq q^{(m-1)/2}(q^{1/2}-1)$).

Consider now a collection of integers $n_{d,u}\geq0$
(for $d,u\geq1$), such that almost all of them are zero, and
that for any $d$,
$$
\label{nd<Ndbis}
\sum_un_{d,u}\leq B_d(F/\F_q).
$$
Suppose the following assumption is satisfied:
$$
	\label{G>2n+3e+g-1}
	\sum_{d,u}n_{d,u}du\geq2ml+3e+g-1,
$$
where the constant $e$ is defined as $e=2$ if $q=2$; $e=1$ if $q=3,4,5$; and $e=0$ if $q\geq7$.
Then 	we have
	$$
	\mu_q(m,l)\leq\sum_{d,u}n_{d,u}\mu_q(d,u).
	$$
\end{Theorem}

Intuitively, the algorithm works as follows: if $x,y$ are two elements in $\F_{q^m}[t]/(t^l)$ to be multiplied, we lift them to functions $f_x,f_y$ in some well-chosen Riemann-Roch spaces of $F$, we evaluate these functions at various places of $F$ with multiplicities (more precisely, $n_{d,u}$ is the number of places of degree $d$ used with multiplicity $u$), we multiply these values locally, and then we interpolate to find the product function $f_xf_y$, from which the product $xy$ is deduced.

Note that this algorithm is a non necessarily symmetric algorithm since $f_x$ and $f_y$ can be lifted in two different Riemann-Roch spaces; so we obtain bounds for $\mu_q(m,l)$, and not for $\mus_q(m,l)$.

	\subsection{New results established in this paper}  

In Section 2, we describe a general method to obtain new uniform bounds for the bilinear complexity of multiplication, by applying the algorithm recalled in Theorem \ref{theobornegenerale} on towers of function fields which satisfy some properties.\\
In Section 3, we recall some results about a completed Garcia-Stichtenoth tower \cite{gast} studied in \cite{ball3} and about the Garcia-Stichtenoth tower introduced in \cite{gast2}. For both towers, we study some of their properties which will be useful in Section 4, to apply the general method on these towers. By doing so, we obtain in Section 4, new uniform bounds on the (asymmetric) bilinear complexity of multiplication in extensions of $\F_2$, of $\F_{q^2}$ and $\F_q$ for any prime power ${q\geq4}$ and of $\F_{p^2}$ and $\F_p$ for any prime ${p\geq3}$, which are the currently known best ones.\\
Last, in Section 5, we turn to the asymptotics of the bilinear complexity as the degree of the extension goes to infinity.
In some cases, the asymptotics of our uniform bounds already improve on previously known results.
But then we also present some (non-uniform) bounds with even better asymptotics, which appear to establish a new present state of the art.

\section{General algorithm used in this paper}

\begin{Lemma}\label{lemmequot}
Let ${d}$ be a positive integer. For any integer ${0 <j\leq d}$  such that  \linebreak[4]${j < \frac{1}{2}\left(q+1+\epsilon(q)\right)}$ if ${q\geq4}$, or ${j \leq\frac{1}{2}q+1}$ if ${q\in\{2,3\}}$, one has
$$
\frac{\mus_q(j)}{j} \leq \frac{\mus_q(d)}{d}.
$$
\end{Lemma}

\begin{Proof}
Suppose that the lemma is false. Then there exists an integer ${0<j < d}$ such that ${j < \frac{1}{2}\left(q+1+\epsilon(q)\right)}$  if ${q\geq4}$ (resp. ${j \leq\frac{1}{2}q+1}$ if ${q\in\{2,3\}}$) and ${\mus_q(j)> \frac{j}{d} \mus_q(d)}$. Two cases can occur:
\begin{itemize}
	\item[\textendash]  either ${j\leq\frac{q}{2}+1}$ (in particular, this is the case if ${q\in \{2,3\}}$), and then we have ${\mus_q(j)> \frac{j}{d}\mus_q(d) \geq \frac{j}{d}(2d-1) > 2j-1}$, 
	\item[\textendash] or ${\frac{q}{2}+1 < j < \frac{1}{2}\left(q+1+\epsilon(q)\right)}$, so ${\mus_q(d)\geq 2d}$ leads to \linebreak[4]${\mus_q(j)> \frac{j}{d}\mus_q(d) \geq 2j}$,
\end{itemize}
so both cases contradict the results recalled in Section \ref{knownresults}.
\qed\\
\end{Proof}

\begin{Proposition}\label{chchexplicite}
Let $q$ be a prime power and $d$ be a positive integer such that any proper divisor $j$ of $d$ satisfies ${j < \frac{1}{2}\left(q+1+\epsilon(q)\right)}$ if ${q\geq4}$, or ${j \leq\frac{1}{2}q+1}$ if ${q\in\{2,3\}}$. Let ${F/\F_q}$ be an algebraic function field of genus ${g\geq2}$ with $N_i$ places of degree $i$ and let $l_i$ be integers such that ${0 \leq l_i \leq N_i}$, for all ${i\vert d}$. Suppose that: 
\begin{enumerate}[(i)]
	\item there exists a place of degree $n$ of $F/\F_q$,
	\item ${\sum_{i\vert d} i(N_i+l_i) \geq 2n + g +\alpha_q}$, where $\alpha_2 =5$, $\alpha_3=\alpha_4=\alpha_5=2$ and $\alpha_q=-1$ for $q>5$.
\end{enumerate}
Then
\begin{equation}\label{borneexplicitedegd1}
\mu_q(n) \leq \frac{2\mus_q(d)}{d}\left(n+\frac{g}{2}\right) + \gamma_{q,d}\sum_{i\vert d}il_i +\kappa_{q,d},
\end{equation}
where ${\gamma_{q,d}:= \max_{i \vert d} \big(\frac{\mu_q(i,2)}{i}\big) - \frac{2\mus_q(d)}{d}}$ and ${\kappa_{q,d}\leq \frac{\mus_q(d)}{d}(\alpha_q+d-1)}$.\\
\end{Proposition}

\begin{Proof}
We apply Theorem \ref{theobornegenerale} with ${n_{i,1}=N_i-l_i}$ and ${n_{i,2}=l_i}$ for any $i \vert d$, and the others ${n_{j,u}=0}$. We choose ${l=1}$ and ${m=n}$ and we get 
\begin{eqnarray*}
\mu_q(n) & \leq & \sum_{i\vert d}\Big(n_{i,1}\mu_q(i)+n_{i,2}\mu_q(i,2)\Big)\\
		& = & \sum_{i\vert d}\Big((N_i-l_i)\mu_q(i)+l_i\mu_q(i,2)\Big)\\
		& \leq & \sum_{i\vert d}\Big((N_i-l_i)\mus_q(i)+l_i\mu_q(i,2)\Big)\\
		& = & \sum_{i\vert d}\Big(\big(N_i+l_i\big)\mus_q(i)+l_i\big(\mu_q(i,2)-2\mus_q(i)\big)\Big)\\
		& = & \sum_{i\vert d}\left(i\big(N_i+l_i\big)\frac{\mus_q(i)}{i}+il_i\left(\frac{\mu_q(i,2)-2\mus_q(i)}{i}\right)\right)
\end{eqnarray*}
so
\begin{eqnarray*}
		\mu_q(n)  & \leq & \frac{\mus_q(d)}{d}\sum_{i\vert d}i\big(N_i+l_i\big) +\sum_{i\vert d}\left(i\big(N_i+l_i\big)\left(\frac{\mus_q(i)}{i}-\frac{\mus_q(d)}{d}\right) \right. \\ & &\left. +~il_i\left(\frac{\mu_q(i,2)-2\mus_q(i)}{i}\right)\right)\\
		 & \leq & \frac{\mus_q(d)}{d}\sum_{i\vert d}i\big(N_i+l_i\big)+\sum_{i\vert d}il_i\left(\frac{\mu_q(i,2)-\mus_q(i)}{i}-\frac{\mus_q(d)}{d}\right) \\ & & +\sum_{i\vert d}iN_i\left(\frac{\mus_q(i)}{i}-\frac{\mus_q(d)}{d}\right)
\end{eqnarray*}
According to Lemma \ref{lemmequot}, we have ${\frac{\mus_q(i)}{i}-\frac{\mus_q(d)}{d} \leq 0}$, so
$$
\sum_{i\vert d}iN_i\left(\frac{\mus_q(i)}{i}-\frac{\mus_q(d)}{d}\right) \leq \sum_{i\vert d}il_i\left(\frac{\mus_q(i)}{i}-\frac{\mus_q(d)}{d}\right)
$$
since ${0\leq l_i\leq N_i}$ for any ${i \vert d}$. Moreover, w.l.o.g we can suppose from (ii) that ${\sum_{i \vert d} i(N_i+l_i) = 2n+g+\alpha_q+k_d}$, with ${k_d \in \{0, \ldots,d-1\}}$. We obtain:
\begin{eqnarray*}
\mu_q(n) & \leq & \frac{\mus_q(d)}{d}(2n+g+\alpha_q+k_d)+\sum_{i\vert d}il_i\left(\frac{\mu_q(i,2)}{i}-\frac{2\mus_q(d)}{d}\right)
\end{eqnarray*}
which gives the result.
\qed \\
\end{Proof}

The two following corollaries are straightforward and give explicit values for Bound (\ref{borneexplicitedegd1}) obtained from the preceding proposition applied for the special cases where $d=1,2$ or~$4$.
\begin{Corollary}\label{chchq>5} Let ${q\geq3}$ be a prime power  and $F/\F_q$ be an algebraic function field of genus ${g\geq2}$ with $N_i$ places of degree $i$ and let $l_i$ be integers such that ${0 \leq l_i \leq N_i}$. If
\begin{enumerate}[(i)]
	\item there exists a place of degree $n$ of $F/\F_q$,
	\item ${N_1+l_1+ 2(N_2+l_2) \geq 2n + g +\alpha_q}$, where $\alpha_3=\alpha_4=\alpha_5=2$ and $\alpha_q=-1$ for $q>5$,
\end{enumerate}
then
$$
\mu_3(n) \leq 3n +\frac{3}{2}g+\frac{3}{2}(l_1+2l_2)+\frac{9}{2},
$$
$$
\mbox{for } q=4\mbox{ or }5, \ \mu_q(n) \leq 3n +\frac{3}{2}g+l_1+2l_2+\frac{9}{2},
$$
and for $q>5$
$$
\mu_q(n) \leq 3n +\frac{3}{2}g+\frac{1}{2}(l_1+2l_2),\mbox{ if }q>5
$$
or  in the special case where $N_2=l_2=0$ (corresponding to $d=1$ in Prop. \ref{chchexplicite})
$$
\mu_q(n) \leq 2n+g+l_1-1.
$$
\end{Corollary}

\begin{Proof}
To apply Proposition \ref{chchexplicite}, let us recall that ${\mus_q(2)=3}$ and ${\mu_q(1,2)\leq3}$ for any prime power $q$. Moreover according to \cite[Example 4.4]{rand3}, one knows that ${\mu_3(2,2) \leq 9}$, ${\mu_q(2,2) \leq 8}$ for ${q =4}$ or $5$ and ${\mu_q(2,2) \leq 7}$ for ${q> 5}$. Hence, we can deduce that
${\gamma_{3,2} \leq \frac{9}{2}-3=\frac{3}{2}}$,
${\gamma_{q,2} \leq \frac{8}{2}-3=1}$ for ${q=4}$ or $5$, and
${\gamma_{q,2} \leq \frac{7}{2}-3=\frac{1}{2}}$ and ${\gamma_{q,1} \leq1}$ for ${q>5}$.
\qed\\
\end{Proof}

\begin{Corollary}\label{chchq=2} Let ${F/\F_2}$ be an algebraic function field of genus ${g\geq2}$ with $N_i$ places of degree $i$ and let $l_i$ be integers such that ${0 \leq l_i \leq N_i}$. If
\begin{enumerate}[(i)]
	\item there exists a place of degree $n$ of $F/\F_2$,
	\item ${\sum_{i\vert 4} i(N_i+l_i) \geq 2n + g +5}$,
\end{enumerate}
then
$$
\mu_2(n) \leq \frac{9}{2}\left(n+\frac{g}{2}\right) + \frac{3}{2}\sum_{i\vert 4}il_i +18.
$$
\end{Corollary}

\begin{Proof}
We recall from \cite[Example 6.1]{chch} that ${\mus_2(4) =9}$ and from \cite[Example 4.4, Lemma 4.6]{rand3} that ${\mu_2(2,2) \leq 9}$ and ${\mu_2(4,2)\leq24}$, which gives ${\gamma_{2,4} \leq \frac{24}{4}-\frac{2\cdot9}{4} = \frac{3}{2}}$.
\qed\\
\end{Proof}

	\subsection{General method to obtain uniform bounds for $\mu_q(n)$}\label{sectmethodegene}
	
We consider a tower ${\mathcal{F}}$  of function fields ${F_i/\F_q}$ of genus ${g(F_i)}$ with ${B_\ell(F_i)}$ places of degree $\ell$. Let $d$ be an integer such that any proper divisor $j$ of $d$ satisfies ${j < \frac{1}{2}\left(q+1+\epsilon(q)\right)}$ if ${q\geq4}$, or ${j \leq\frac{1}{2}q+1}$ if ${q\in\{2,3\}}$. \\
Suppose there exists an integer $N$ such that, for all ${n\geq N}$, there is an integer~$k(n)$ for which:
\begin{enumerate}[(A)]
	\item \label{hypgene1} ${\sum_{j \vert d} j B_j(F_{k(n)+1}) \geq 2n +g(F_{k(n)+1})+ \alpha_q}$ and $B_{n}(F_{k(n)+1}) >0$,
	\item \label{hypgene1bis} ${\sum_{j \vert d} j B_j(F_{k(n)}) < 2n +g(F_{k(n)})+ \alpha_q}$ but $B_n(F_{k(n)})>0$,
	\item \label{hypgene3} $g(F_{k(n)}) \geq 2$ (so $g(F_{k(n)+1}) \geq 2$),
	\item \label{hypgene4} ${\Delta g_{k(n)} :=g(F_{k(n)+1})-g(F_{k(n)}) \geq \lambda D_{k(n)}}$ with $\lambda:= \frac{d\gamma_{q,d}}{\mus_q(d)}$,
	\item \label{hypgene5} ${\sum_{j\vert d} jB_j(F_{k(n)}) \geq D_{k(n)}}$,
\end{enumerate}
where ${\alpha_q}$ is as in Proposition \ref{chchexplicite} and $D_{k(n)}$ is chosen to satisfy (\ref{hypgene4}) and (\ref{hypgene5}), and is fixed for the tower $\mathcal{F}$.\\
 We also set 
 $$
 n_0^{l}:= \sup \Big \{m \in \mathbb{N} \ \Big  \vert \ \sum_{j\vert d} jB_j(F_{l}) \geq 2m + g(F_{l}) + \alpha_q \Big\}.
 $$
 Note that for the integer $n_0^{k(n)}$, the following holds:
\begin{equation}\label{nbplacesFk}
\sum_{j\vert d} jB_j(F_{k(n)}) + 2\left(n-n_0^{k(n)}\right) \geq 2n + g(F_{k(n)}) + \alpha_q.
\end{equation}

Now, fix an integer $n\geq N$ and let $k:=k(n)$ satisfying Hypotheses (\ref{hypgene1}) to~(\ref{hypgene5}).\\
To multiply in $\F_{q^n}$, one has the following alternative:
\begin{enumerate}[(a)]
	\item\label{cas1algogene}  apply the algorithm on the step $F_{k+1}$, with $B_j(F_{k+1})$ places of degree $j$ for any ${j \vert d}$, all of them used with multiplicity 1; this is possible according to (\ref{hypgene1}) and (\ref{hypgene3}). In this case, Proposition~\ref{chchexplicite} gives the following bound for $\mu_q(n)$:
	\begin{equation}\label{borneFk+1}
	\mu_q(n) \leq  \frac{2\mus_q(d)}{d}\left(n+\frac{g(F_{k+1})}{2}\right) +\frac{\mus_q(d)}{d}(\alpha_q+d-1),
	\end{equation}
	\item\label{cas2algogene}  apply the algorithm on the step $F_k$, with $B_j(F_k)$ places of degree $j$ of which $l_j$ used with multiplicity 2 and the remaining with multiplicity 1, for any ${j \vert d}$, where the integers $l_j \leq B_j(F_k)$ satisfy ${\sum_{j \vert d} l_j\geq 2(n-n_0^k)}$; for such integers $l_j$, we can apply Proposition \ref{chchexplicite} according to (\ref{hypgene1bis}) and~(\ref{nbplacesFk}). In particular, if ${2(n-n_0^k) + d-1\leq \sum_{j \vert d} jB_j(F_k)}$, then we can choose the integers $l_j$ such that ${\sum_{j \vert d} jl_j = 2(n-n_0^k) + \epsilon}$ for some ${\epsilon \in \{0, \dotsc,d-1\}}$, and this is a suitable choice. In this case, Proposition~\ref{chchexplicite} gives:
	\begin{equation}\label{borneFk}
	\mu_q(n) \leq \frac{2\mus_q(d)}{d}\left(n+\frac{g(F_k)}{2}\right) + \gamma_{q,d}\sum_{i\vert d}il_i +\frac{\mus_q(d)}{d}(\alpha_q+d-1).
	\end{equation}
\end{enumerate}
Note that we can rewrite (\ref{borneFk+1}) as follow:
$$
\mu_q(n) \leq  \frac{2\mus_q(d)}{d}\left(n+\frac{g(F_{k})}{2}\right) + \frac{\mus_q(d)}{d}\Delta g_k+\frac{\mus_q(d)}{d}(\alpha_q+d-1)
$$
which makes clear that if ${\gamma_{q,d} \sum_{i \vert d} i l_i <  \frac{\mus_q(d)}{d} \Delta g_k}$, then Case (\ref{cas2algogene}) gives a better bound then Case (\ref{cas1algogene}).\\
So if ${2(n-n_0^k) + d-1< D_k}$ , then  we can proceed as in Case (\ref{cas2algogene}) since according to Hypothesis (\ref{hypgene5}) we can choose ${\epsilon \in \{0, \ldots,d-1\}}$ and $l_j$ for ${j \vert d}$ such that ${\sum_{j \vert d} jl_j = 2(n-n_0^k) + \epsilon}$. Moreover, we have
$$
\frac{d\gamma_{q,d}}{\mus_q(d)}(2(n-n_0^k) + d-1) <  \Delta g_k
$$
from Hypothesis (\ref{hypgene4}), so ${\gamma_{q,d}\left(2(n-n_0^k) + \epsilon\right) <  \frac{\mus_q(d)}{d} \Delta g_k}$ which means that the bound obtained from Case (\ref{cas2algogene}) is sharper.
\\
For ${x\in \R^+}$, ${x\geq N}$, such that ${\sum_{j \vert d} j B_j(F_{k+1}) \geq 2\left[x\right] + g(F_{k+1})+ \alpha_q}$ and \linebreak[4]${\sum_{j \vert d} j B_j(F_{k+1}) < 2\left[ x \right] + g(F_k)+ \alpha_q}$, we define the function $\Phi_k(x)$ as follows:
$$
\Phi_k(x) = \left\{\begin{array}{l}
				\frac{2\mus_q(d)}{d}\left(x+\frac{g(F_k)}{2}\right) + \gamma_{q,d}\big(2(x-n_0^k) + d -1\big)+\frac{\mus_q(d)}{d}(\alpha_q+d-1),\\
				\mbox{ \hspace{16em} if ${2(x-n_0^k) + d-1< D_k}$.}\\
				\frac{2\mus_q(d)}{d}\left(x+\frac{g(F_{k+1})}{2}\right) +\frac{\mus_q(d)}{d}(\alpha_q+d-1)\mbox{, else.}
			      \end{array} \right.
$$
that is to say:
$$
\Phi_k(x) = \left\{\begin{array}{l}
				\left(\frac{2\mus_q(d)}{d}+2\gamma_{q,d}\right)(x-n_0^k)+\frac{\mus_q(d)}{d}\left(2n_0^k+g(F_k)+\alpha_q+d-1\right), \\
				\mbox{ \hspace{16em} if ${2(x-n_0^k) + d-1< D_k}$.} \\
				\frac{2\mus_q(d)}{d}(x-n_0^k)+\frac{\mus_q(d)}{d}\left(2n_0^k+g(F_{k+1})+\alpha_q+d-1\right)\mbox{, else.}
			      \end{array} \right.
$$
We define the function $\Phi$ for all ${x\geq N}$ as the minimum of the functions $\Phi_i$ for which $x$ is in the domain of $\Phi_i$. This function is piecewise linear with two kinds of pieces: those which have slope $\frac{2\mus_q(d)}{d}$ and those which have slope $\frac{2\mus_q(d)}{d}+2\gamma_{q,d}$. Moreover, the graph of the function $\Phi$ lies below any straight line that lies above all the points ${\big( n_0^i+\frac{1}{2}(D_i-d+1), \Phi(n_0^i+\frac{1}{2}(D_i-d+1))\big)}$, since these are the \textit{vertices} of the graph. Let ${X:= n_0^i+\frac{1}{2}(D_i-d+1)}$, then
\begin{eqnarray*}
\Phi(X) & = & \frac{2\mus_q(d)}{d}\left(X+\frac{g(F_{i+1})}{2}\right) +\frac{\mus_q(d)}{d}(\alpha_q+d-1)\\
	&  = & \frac{2\mus_q(d)}{d}\left(1+\frac{g(F_{i+1})}{2X}\right)X +\frac{\mus_q(d)}{d}(\alpha_q+d-1).
\end{eqnarray*}
If we can give a bound for $\Phi(X)$ which is independent of $i$, then it will provide a bound for ${\mu_q(n)}$ for all ${n\geq N}$, since ${\mu_q(n) \leq\Phi(n)}$.

\section{Good sequences of function fields}\label{sequences}

\subsection{Garcia-Stichtenoth tower of Artin-Schreier algebraic function field extensions}\label{sectiondefGS}

We present now a modified Garcia-Stichtenoth's tower (cf. \cite{gast}, \cite{ball3}, \cite{baro1}) having good properties. 
Let  us consider a finite field~$\F_{q^2}$ with ${q=p^r\geq4}$ and $r$ an integer. 
We consider the Garcia-Stichtenoth's elementary abelian tower $T_1$ over $\F_{q^2}$ 
constructed in \cite{gast} and defined by the sequence $(F_1, F_2, F_3,\ldots)$ where 
$$F_{k+1}:=F_{k}(z_{k+1})$$ 
and $z_{k+1}$ 
satisfies the equation: 
$$z_{k+1}^q+z_{k+1}=x_k^{q+1}$$ 
with 
$$x_k:=z_k/x_{k-1}\mbox{ in } F_k\mbox{ (for $k\geq2$).}$$
Moreover $F_1:=\F_{q^2}(x_1)$ is the rational function field over $\F_{q^2}$ 
and  $F_2$ the Hermitian function field over $\F_{q^2}$.
Let us denote by $g_k$ the genus of $F_k$, we recall the following \textsl{formulae}:
\begin{equation}\label{genregs}
g_k = \left \{ \begin{array}{ll}
		q^k+q^{k-1}-q^\frac{k+1}{2} - 2q^\frac{k-1}{2}+1 & \mbox{if } k \equiv 1 \mod 2,\\
		q^k+q^{k-1}-\frac{1}{2}q^{\frac{k}{2}+1} - \frac{3}{2}q^{\frac{k}{2}}-q^{\frac{k}{2}-1} +1& \mbox{if } k \equiv 0 \mod 2.
		\end{array} \right .
\end{equation}
Let us consider the completed Garcia-Stichtenoth tower 
$$T_2=F_{1,0}\subseteq F_{1,1}\subseteq \cdots \subseteq F_{1,r} = F_{2,0}\subseteq F_{2,1} \subseteq \cdots \subseteq F_{2,r} \subseteq \cdots $$ 
considered in \cite{ball3} such that   
$F_k \subseteq F_{k,s} \subseteq F_{k+1}$ for any integer $s \in \{0,\ldots,r\}$,  
with $F_{k,0}=F_k$ and $F_{k,r}=F_{k+1}$. Recall that each extension $F_{k,s}/F_k$ is Galois of degree $p^s$ 
with full constant field $\F_{q^2}$. 
Now, we consider the tower studied in \cite{baro1}
$$T_3=G_{1,0} \subseteq G_{1,1} \subseteq \cdots \subseteq G_{1,r} = G_{2,0}\subseteq G_{2,1}\subseteq \cdots \subseteq G_{2,r} \subseteq \cdots $$
defined over the constant field $\F_q$ and related to
the tower $T_2$ by
$$F_{k,s}=\F_{q^2}G_{k,s} \quad \mbox{for all $k$ and $s$,}$$
namely $F_{k,s}/\F_{q^2}$ is the constant field extension of $G_{k,s}/\F_q$. Note that the tower $T_3$ is well defined 
by \cite{baro1} and \cite{balbro}. Moreover, we have the following result:

\begin{Proposition}\label{subfieldFq}
Let ${q = p^r \geq4}$ be a prime power. For all integers $k \geq 1$ and ${s \in \{0, \ldots,r\}}$, there exists a step $F_{k,s}/\F_{q^2}$ (respectively $G_{k,s}/\F_q$) with genus $g_{k,s}$ and $N_{k,s}$  places of degree one in $F_{k,s}/\F_{q^2}$ (respectively  \linebreak[4]${N_{k,s}:=B_1(G_{k,s}/\F_q)+2B_2(G_{k,s}/\F_q)}$ where ${B_i(G_{k,s}/\F_q)}$ denote the number of places of degree $i$ in ${G_{k,s}/\F_q}$) such that:
\begin{enumerate}[(1)]
	\item $F_k \subseteq F_{k,s} \subseteq F_{k+1}$, where we set ${F_{k,0}:=F_k}$ and ${F_{k,r}:=F_{k+1}}$,\\
	(respectively $G_k \subseteq G_{k,s} \subseteq G_{k+1}$, with ${G_{k,0}:=G_k}$ and ${G_{k,r}:=G_{k+1}}$),
	\item $\big( g_k-1 \big)p^s +1 \leq g_{k,s} \leq \frac{g_{k+1}}{p^{r-s}} +1$,
	\item $N_{k,s} \geq (q^2-1)q^{k-1}p^s$.
\end{enumerate} 
\end{Proposition}

Now, we are interested to search the descent of the definition field of the tower $T_2/\F_{q^2}$ from $\F_{q^2}$ to  $\F_{p}$ if it is possible. In fact, one cannot establish a general result but one can prove that it is possible in the case of characteristic~$2$ which is given by the following result obtained in \cite{baro4}. 

\begin{Proposition}\label{proptourF2}
Let $p=2$. If $q=p^2$, the descent of the definition field of the tower $T_2/\F_{q^2}$ from $\F_{q^2}$ to  $\F_{p}$ is possible. More precisely,  there exists a tower $T_4/\F_p$ defined over $\F_{p}$ given by a sequence:
$$
T_4/\F_p=H_{1,0} \subseteq H_{1,1} \subseteq H_{1,2} = H_{2,0}\subseteq H_{2,1}\subseteq H_{2,2} = H_{3,0} \subseteq \cdots
$$
defined over the constant field $\F_p$ and related to the towers $T_1/\F_{q^2}$ and $T_2/\F_q$ by
$$
F_{k,s}=\F_{q^2}H_{k,s} \mbox{ for all $k$ and $s=0,1,2$},
$$
$$
G_{k,s}=\F_{q}H_{k,s} \mbox{ for all $k$ and $s=0,1,2$},
$$
namely $F_{k,s}/\F_{q^2}$ is the constant field extension of $G_{k,s}/\F_q$ and $H_{k,s}/\F_p$ and $G_{k,s}/\F_q$ is the constant field extension of $H_{k,s}/\F_p$. 
\end{Proposition}

Moreover, from \cite{baro4}, the following properties holds for this tower ${T_3/\F_p}$:

\begin{Proposition}\label{subfieldF2}
Let $q=p^2=4$. For any integers ${k\geq1}$ and ${s \in \{0,1,2\}}$, the algebraic function field $H_{k,s}/\F_{p}$ in the tower $T_3/\F_p$ with genus \linebreak[4]${g_{k,s}:=g(H_{k,s}/\F_p)}$ and ${B_{i}(H_{k,s}/\F_p)}$ places of degree $i$, is such that:
\begin{enumerate}[(1)]
	\item ${H_k/\F_p \subseteq H_{k,s} /\F_p \subseteq H_{k+1}/\F_p}$ with ${H_{k,0}=H_k}$ and ${H_{k,2}=H_{k+1}}$,
	\item ${g_{k,s}\leq \frac{g_{k+1}}{p^{2-s}}+1}$ with ${g_{k+1}\leq q^{k+1}+q^{k}}$,   
	\item ${B_{1}(H_{k,s}/\F_p)+2B_{2}(H_{k,s}/\F_p) +4B_{4}(H_{k,s}/\F_p)\geq (q^2-1)q^{k-1}p^{s}}$.
\end{enumerate}
\end{Proposition}

	\subsection{Garcia-Stichtenoth tower of Kummer function field extensions}\label{sectiondefGSR}

In this section we present a Garcia-Stichtenoth's tower (cf. \cite{bach}) having good properties.
Let $\F_q$ be a finite field of characteristic $p\geq3$.
Let us consider the tower $T$ over $\F_q$ which is defined recursively by the following equation, studied in \cite{gast2}:
$$y^2=\frac{x^2+1}{2x}.$$
The tower $T/\F_q$ is represented by the sequence of function fields\linebreak[4]${(L_0, L_1, L_2, \ldots)}$ where ${L_n = \F_q(x_0, x_1, \ldots, x_n)}$ and ${x_{i+1}^2=(x_i^2+1)/2x_i}$ holds for each ${i\geq 0}$.
Note that $L_0$ is the rational function field.
For any prime number ${p \geq 3}$, the tower $T/\F_{p^2}$ is asymptotically optimal over the field ${\F_{p^2}}$, i.e. ${T/\F_{p^2}}$ reaches the Drinfeld-Vl\u{a}du\c{t} bound.
Moreover, for any integer $k$, ${L_k/\F_{p^2}}$ is the constant field extension of ${L_k/\F_p}$.

From \cite{bach}, we know that the genus $g(L_k)$ of the steps ${L_k/\F_{p^2}}$ and ${L_k/\F_{p}}$ is given by:
\begin{equation}\label{genregsr}
g(L_k) = \left \{ \begin{array}{ll}
		2^{k+1}-3\cdot 2^\frac{k}{2}+1 & \mbox{if } k \equiv 0 \mod 2,\\
		2^{k+1} -2\cdot 2^\frac{k+1}{2}+1& \mbox{if } k \equiv 1 \mod 2.
		\end{array} \right .
\end{equation}
and that the following bounds hold for the number of rational places  in $L_k$ over $\F_{p^2}$ and for  the number of places of degree one and two over $\F_p$:
\begin{equation}\label{nbratplgsr}
B_1(L_k/\F_{p^2}) \geq 2^{k+1}(p-1)
\end{equation}
and
\begin{equation}\label{nbpldeg12gsr}
B_1(L_k/\F_p) +2B_2(L_k/\F_p) \geq 2^{k+1}(p-1).
\end{equation}

	\subsection{Some preliminary results} \label{sectionusefull}

Here we establish some technical results about genus and number of places of each step of the towers ${T_2/\F_{q^2}}$, ${T_3/\F_q}$, ${T_4/\F_2}$, ${T/\F_{p^2}}$ and ${T/\F_p}$ defined in Sections \ref{sectiondefGS} and \ref{sectiondefGSR}. These results will allow us to determine a suitable step of the tower to apply the algorithm on. 

		\subsubsection{About the Garcia-Stichtenoth's tower  of Artin-Schreier extensions}
In this section, ${q=p^r}$ is a power of the prime $p$. We denote by $g_{k,s}$ the genus of the corresponding steps of the towers ${T_2/\F_{q^2}}$, ${T_3/\F_q}$ and ${T_4/\F_2}$; recall that ${g_k=g_{k,0}=g_{k-1,r}}$. We also set 
$$
\Delta g_{k,s} := g_{k,s+1}-g_{k,s}.
$$

\begin{Lemma}\label{lemme_genre}
Let ${q\geq4}$. We have the following bounds for the genus of each step of the towers ${T_2/\F_{q^2}}$,  ${T_3/\F_q}$ and ${T_4/\F_2}$ (we set ${q=4}$ and ${p=r=2}$ in the special case of this tower):
\begin{enumerate}[i)]
	\item $g_k> q^k$ for all ${k\geq 4}$,\\
		moreover for the tower ${T_4/\F_2}$, one has ${g_k > pq^{k-1}}$ for all ${k\geq3}$, 
	\item $g_k \leq q^{k-1}(q+1) - \sqrt{q}q^\frac{k}{2}$,
	\item $g_{k,s} \leq q^{k-1}(q+1)p^s$ for all ${k\geq 0}$ and ${s \in \{0,\ldots,r\}}$,
	\item $g_{k,s} \leq \frac{q^k(q+1)-q^\frac{k}{2}(q-1)}{p^{r-s}}$ for all ${k\geq 2}$ and ${s\in \{0,\ldots,r\}}$.
\end{enumerate}
\end{Lemma}

\begin{Proof}
\begin{enumerate}[i)]
	\item According to Formula (\ref{genregs}), we know that if ${k \equiv 1 \mod 2}$, then 
$$
g_k = q^k+q^{k-1}-q^\frac{k+1}{2} - 2q^\frac{k-1}{2}+1 = q^k+q^\frac{k-1}{2}(q^\frac{k-1}{2} - q - 2) +1.
$$
Since ${q>3}$ and ${k \geq 4}$, we have ${q^\frac{k-1}{2} - q - 2 >0}$, thus ${g_k>q^k}$.\\
Else if ${k \equiv 0 \mod 2}$, then 
$$
g_k = q^k+q^{k-1}-\frac{1}{2}q^{\frac{k}{2}+1} - \frac{3}{2}q^{\frac{k}{2}}-q^{\frac{k}{2}-1} +1 = q^k+q^{\frac{k}{2}-1}(q^\frac{k}{2}-\frac{1}{2}q^{2} - \frac{3}{2}q-1)+1.
$$
Since ${q>3}$ and ${k\geq 4}$, we have ${q^\frac{k}{2}-\frac{1}{2}q^{2} - \frac{3}{2}q-1>0}$, thus ${g_k>q^k}$.\\
Hence, the second bound for the tower ${T_4/\F_2}$ is already proved for ${k\geq4}$, and for ${k=3}$, one has ${g_3-pq^2= q^3-2q+1-pq^2=25}$ so this bound holds also for ${k=3}$. 
	\item It follows from Formula (\ref{genregs}) since for all $k\geq 1$ we have ${2q^\frac{k-1}{2} \geq 1}$ which works out for odd $k$ cases and ${\frac{3}{2}q^\frac{k}{2}+q^{\frac{k}{2}-1}\geq 1}$ which works out for even $k$ cases, since ${\frac{1}{2}q\geq \sqrt{q}}$.
	\item If ${s=r}$, then according to Formula (\ref{genregs}), we have 
$$
g_{k,s} = g_{k+1}\leq q^{k+1}+q^{k} = q^{k-1}(q+1)p^s.
$$
Else, ${s<r}$ and Proposition \ref{subfieldFq} says that ${g_{k,s} \leq \frac{g_{k+1}}{p^{r-s}}+1}$. Moreover, since ${q^\frac{k+2}{2}\geq q}$ and ${\frac{1}{2}q^{\frac{k+1}{2}+1}\geq q}$, we obtain ${g_{k+1}\leq q^{k+1} + q^k - q + 1}$ from Formula (\ref{genregs}). Thus, we get 
\begin{eqnarray*}
g_{k,s} & \leq & \frac{q^{k+1} + q^k - q + 1}{p^{r-s}} +1\\
	    &  = & q^{k-1}(q+1)p^s - p^s + p^{s-r} + 1\\
	    & \leq & q^{k-1}(q+1)p^s + p^{s-r}\\
	    & \leq & q^{k-1}(q+1)p^s \  \mbox{ since ${0 \leq p^{s-r} <1}$ and ${g_{k,s} \in \N}$}.
\end{eqnarray*}
	\item It follows from ii) since Proposition \ref{subfieldFq} gives ${g_{k,s} \leq \frac{g_{k+1}}{p^{r-s}}+1}$, so \linebreak[4]${g_{k,s} \leq  \frac{q^{k}(q+1) - \sqrt{q}q^\frac{k+1}{2}}{p^{r-s}} +1}$ which gives the result since ${p^{r-s} \leq q^\frac{k}{2}}$ for all ${k\geq2}$.
\end{enumerate}
\qed\\
\end{Proof}

Now we set ${N_{k,s}:= B_1(F_{k,s}/\F_{q^2}) =  B_1(G_{k,s}/\F_q)+2B_2(G_{k,s}/\F_q)}$.

\begin{Lemma}{\label{lemme_deltaFq}}
Let ${D_{k,s}:=(p-1)p^sq^k}$. For any ${k\geq1}$ and ${s\in \{0,\ldots,r-1\}}$, one has:
\begin{enumerate}[i)]
	\item $\Delta g_{k,s} \geq D_{k,s}$ if ${k\geq4}$,
	\item $N_{k,s} \geq D_{k,s}$.
\end{enumerate}
\end{Lemma}

\begin{Proof}
\begin{enumerate}[i)]
	\item From  Hurwitz Genus Formula, one has ${g_{k,s+1}-1 \geq p(g_{k,s}-1)}$, so \linebreak[4]${g_{k,s+1}-g_{k,s} \geq (p-1)(g_{k,s}-1)}$. Applying $s$ more times Hurwitz Genus Formula, we get ${g_{k,s+1}-g_{k,s} \geq (p-1)p^s(g_k-1)}$. Thus we have\linebreak[4]${g_{k,s+1}-g_{k,s} \geq (p-1)p^sq^k}$, from Lemma \ref{lemme_genre} i) since $q>3$ and $k\geq 4$.
	\item According to Proposition \ref{subfieldFq}, one has
\begin{eqnarray*}
	N_{k,s} & \geq & (q^2-1)q^{k-1}p^s \\
		    & = & (q+1)(q-1)q^{k-1}p^s \\
		    & \geq & (q-1)q^kp^s\\
		    & \geq & (p-1)q^kp^s \mbox{.}
\end{eqnarray*}
\end{enumerate}
\qed\\
\end{Proof}

\begin{Lemma}\label{lemme_bornesupFq}
For all ${k \geq 1}$ and ${s\in \{0, \ldots, r\}}$, one has
$$
\sup \big \{ n \in \N \; | \;  N_{k,s} \geq 2n  +g_{k,s} -1 \big \} \geq \frac{1}{2}(q+1)q^{k-1}p^s(q-2)+\frac{1}{2}.
$$
\end{Lemma}

\begin{Proof}
From Proposition \ref{subfieldFq} and Lemma \ref{lemme_genre} iii), we get 
\begin{eqnarray*}
N_{k,s} -g_{k,s} +1 & \geq & (q^2-1)q^{k-1}p^s - q^{k-1}(q+1)p^s +1 \\
				& = & (q+1)q^{k-1}p^s\big((q-1) -1\big) +1.
\end{eqnarray*}
\qed \\
\end{Proof}

Now we recall similar technical results about genus and number of places of each step of the tower $T_4/\F_2$ defined in Section \ref{sectiondefGS}. In order to simplify the presentation, we still use the variables $p$ and~$q$.

\begin{Lemma}\label{lemme_delta_F2}
 Let $q=p^2=4$. For all $k\geq 1$ and ${s \in \{0, 1\}}$, we set ${D_{k,s}:=\frac{3}{2}p^{s+1}q^{k-1}}$. Then we have
\begin{enumerate}[i)]
	\item $\Delta g_{k,s}  \geq \lambda D_{k,s}$, with ${\lambda := \frac{4\gamma_{2,4}}{\mus_2(4)}\leq\frac{3}{2}}$ (see Section \ref{sectmethodegene}),
	\item $B_1(H_{k,s}/\F_p) + 2B_2(H_{k,s}/\F_p) + 4B_4(H_{k,s}/\F_p) \geq D_{k,s}$.
\end{enumerate}
\end{Lemma}

\begin{Proof} 
\begin{enumerate}[i)]
	\item We apply Genus Hurwitz Formula as in the proof of Lemma \ref{lemme_deltaFq} to obtain ${g_{k,s+1}-g_{k,s}\geq (p-1)p^s(g_k-1)}$, so we get ${\Delta g_{k,s}\geq (p-1)p^{s+1}q^{k-1}}$ from Lemma \ref{lemme_genre}~i) for ${k\geq3}$, which gives the results. For ${k=1}$ and $2$, we check that the result is still valid since ${g_1=0}$, ${g_{1,1}=2}$, ${g_2=6}$, ${g_{2,1}=23}$ and ${g_3=57}$.
	\item It is obvious since ${q^2-1 > \frac{3}{2}p}$ and since  from Proposition~\ref{subfieldF2} we have ${B_1(H_{k,s}/\F_2) + 2B_2(H_{k,s}/\F_2) + 4B_4(H_{k,s}/\F_2) \geq (q^2-1)q^{k-1}p^s}$.
\end{enumerate}
\qed\\
\end{Proof}

\begin{Lemma}\label{lemme_bornesup_F2}
Let ${q=p^2=4}$. For all ${k \geq 1}$ and ${s\in \{0, 1, 2\}}$, we have
$$
\sup \Big \{ n \in \N \; \Big \vert \;  \sum_{i=1,2,4} iB_i(H_{k,s}/\F_2) \geq 2n+g_{k,s} +5 \Big \} \geq 5p^sq^{k-1}-\frac{5}{2}.
$$
\end{Lemma}

\begin{Proof}
From Proposition \ref{subfieldF2} and Lemma \ref{lemme_genre} iii), we get 
\begin{eqnarray*}
 \sum_{i=1,2,4} iB_i(H_{k,s}/\F_2) -g_{k,s} - 5 & \geq & (q^2-1)q^{k-1}p^s - q^{k-1}(q+1)p^s-5 \\
						   & = & p^sq^{k-1}(q+1)(q-2) -5
\end{eqnarray*}
thus we get the result since ${q=4}$.
\qed
\end{Proof}

		\subsubsection{About the Garcia-Stichtenoth's tower of Kummer extensions}
In this section, $p$ is an odd prime. We denote by $g_k$ the genus of the step $L_k$ and we fix 
$$
{N_k:= B_1(L_k/\F_{p^2})=B_1(L_k/\F_p)+2B_2(L_k/\F_p)}
$$
and
$$
\Delta g_k := g_{k+1} - g_k.
$$
The following lemma is straightforward according to Formulae ~(\ref{genregsr}):
\begin{Lemma}\label{lemme_genregsr}
These two bounds hold for the genus of each step of the towers $T/\F_{p^2}$ and $T/\F_p$:
\begin{enumerate}[i)]
	\item $g_k \leq 2^{k+1}-2\cdot2^\frac{k+1}{2}+1$,
	\item $g_k \leq 2^{k+1}$.
\end{enumerate}
\end{Lemma}

\begin{Lemma}\label{lemme_deltagsr}
For all $k\geq 0$,  one has ${N_k \geq \Delta g_k \geq 2^{k+1}- 2^\frac{k+1}{2}}$.
\end{Lemma}

\begin{Proof} If $k$ is even then ${\Delta g_k = 2^{k+1}-2^\frac{k}{2}}$, else ${\Delta g_k = 2^{k+1}-2^\frac{k+1}{2}}$ so the second equality holds trivially. Moreover, since ${p\geq 3}$, the first one follows from Bounds (\ref{nbratplgsr}) and (\ref{nbpldeg12gsr}) which gives ${N_k \geq 2^{k+2}}$.
\qed
\end{Proof}

\begin{Lemma}\label{lemme_bornesupgsr}
Let $L_k$ be a step of one of the towers $T/\F_{p^2}$ or $T/\F_p$. One has:
$$
\sup \big \{ n \in \N \; | \;  N_k \geq 2n +g_k -1\big \} \geq 2^{k}(p-2)+2^\frac{k+1}{2} \mbox{, if } p>5
$$
and 
$$
\sup \big \{ n \in \N \; | \;  N_k \geq 2n +g_k +2\big \} \geq 2^{k}(p-2)+2^\frac{k+1}{2}-1 \mbox{, if } p=5 \mbox{ or }3.
$$
\end{Lemma}

\begin{Proof}
From Bounds (\ref{nbratplgsr}) and (\ref{nbpldeg12gsr}) for $N_k$ and Lemma \ref{lemme_genregsr} i), we get 
\begin{eqnarray*}
N_k - g_k +1 & \geq & 2^{k+1}(p-1) -(2^{k+1}-2\cdot2^\frac{k+1}{2}+1) +1\\
		     & = & 2^{k+1}(p-2) + 2\cdot2^\frac{k+1}{2}.
\end{eqnarray*}
Similarly, we get
\begin{eqnarray*}
N_k - g_k -2 & \geq & 2^{k+1}(p-1) -(2^{k+1}-2\cdot2^\frac{k+1}{2}+1) -2\\
		     & = & 2^{k+1}(p-2) + 2\cdot2^\frac{k+1}{2} - 3
\end{eqnarray*}
which gives the result for $p=5$ or $3$.
\qed
\end{Proof}

	\subsection{Existence of a good step in each tower}

The following lemmas prove the existence of a \guillemotleft~good~\guillemotright~step of the towers defined in Sections \ref{sectiondefGS} and \ref{sectiondefGSR}, that is to say a step that will be optimal for the bilinear complexity of multiplication in a degree $n$ extension of $\F_q$, for any integer $n$.

\begin{Lemma}\label{lemme_placedegn_Fq2}
Let $n \geq \frac{1}{2}\left(q^2+1+\epsilon(q^2)\right)$ be an integer. If ${q=p^r\geq 4}$, then there exists a step $F_{k,s}/\F_{q^2}$ of the tower $T_2/\F_{q^2}$ such that the following conditions are verified:
\begin{enumerate}[(1)]
	\item there exists a place of $F_{k,s}/\F_{q^2}$ of degree $n$,
	\item $B_1(F_{k,s}/\F_{q^2}) \geq 2n + g_{k,s}-1$.
\end{enumerate}
Moreover, the first step for which both Conditions (1) and (2) are verified is the first step for which (2) is verified.
\end{Lemma}

\begin{Proof}
Note that $n \geq 13$ since $q\geq4$ and ${n \geq \frac{1}{2}(q^2+1+2q) \geq 12.5}$. First, we prove that for ${1 \leq k \leq n-2}$ and ${s \in \{0, \ldots, r\}}$, there exists a  \linebreak[4]place of $F_{k,s}/\F_{q^2}$ of degree $n$. Indeed, for such an integer $k$, one has \linebreak[4]${q^{n-k-1}\geq q>2\times \frac{5}{3}\geq 2\frac{q+1}{q-1}}$, so ${q^{n-k}p^{-s}>2\frac{q+1}{q-1}}$ since ${1\geq p^{-s}\geq q}^{-1}$, which gives  ${2q^{k-1}(q+1)p^s < q^{n-1}(q-1)}$. Thus Lemma~\ref{lemme_genre}~iii) implies that  \linebreak[4]${2g_{k,s}+1 \leq q^{n-1}(q-1)}$, which ensures that  there exists a  place of $F_{k,s}/\F_{q^2}$ of degree $n$.
On the other hand, we prove that for ${k\geq K(n)+1}$, with ${K(n):=\log_q\left(\frac{2n}{(q+1)(q-2)}\right)}$, Condition (2) is satisfied. Indeed, for such integers $k$, one has ${\frac{2n}{(q+1)(q-2)} \leq q^{k-1}}$, so ${2n-1 \leq q^{k-1}(q+1)(q-2)p^s}$. Hence, one gets ${2n+q^{k-1}(q+1)p^s-1\leq (q^2-1)q^{k-1}p^s}$, which gives the result according to Lemma~\ref{lemme_genre}~iii) and Proposition~\ref{subfieldFq}~(3). To conclude, note that there exists at least one step $F_{k,s}/\F_{q^2}$ satisfying both Conditions (1) and (2) since for ${n\geq13}$ and ${q\geq4}$, ${n-K(n)-3\geq 13 -(\log_4(2\cdot13))-3>1}$. Moreover, remark that Condition (1) is satisfied from the step $F_{1,0}/\F_{q^2}$, so  the first step for which both Conditions (1) and (2) are verified is the first step for which (2) is verified.
\qed\\
\end{Proof}

This is a similar result for the tower $T_3/\F_q$:

\begin{Lemma}\label{lemme_placedegn_Fq}
Let $n \geq \frac{1}{2}\left(q+1+\epsilon(q)\right)$ be an integer. If $q=p^r>5$, then there exists a step $G_{k,s}/\F_q$ of the tower $T_3/\F_q$ such that the following conditions are verified:
\begin{enumerate}[(1)]
	\item there exists a place of $G_{k,s}/\F_q$ of degree $n$, 
	\item $B_1(G_{k,s}/\F_q)+2B_2(G_{k,s}/\F_q) \geq 2n + g_{k,s}-1$.
\end{enumerate}
Moreover, the first step for which both Conditions (1) and (2) are verified is the first step for which (2) is verified.
\end{Lemma}

\begin{Proof} 
Here we have  $n \geq 7$ since $q\geq7$ and ${n \geq \frac{1}{2}(q+1+\epsilon(q)) \geq 6.5}$. First, we prove that for ${1 \leq k \leq \frac{n}{2}-2}$ and ${s \in \{0, \ldots, r\}}$, there exists a  \linebreak[4]place of $G_{k,s}/\F_{q}$ of degree $n$, by showing that ${2g_{k,s}+1\leq q^\frac{n-1}{2}(\sqrt{q}-1)}$. Indeed, the function ${q\mapsto \frac{\sqrt{q}-1}{q+1} \cdot q^{\frac{n-1}{2}-k}}$ is increasing, so one has \linebreak[4]${\frac{\sqrt{q}-1}{q+1}\cdot q^{\frac{n-1}{2}-k}\geq \frac{\sqrt{7}-1}{8} \cdot 7^{\frac{n-1}{2}-k}}$ since ${q\geq 7}$. Thus for any  ${k\leq\frac{n}{2}-2}$, we get ${\frac{\sqrt{q}-1}{q+1} \cdot q^{\frac{n-1}{2}-k}\geq \frac{7^\frac{3}{2}(\sqrt{7}-1)}{8}>2}$. It follows that ${2q^k(q+1)< q^{\frac{n-1}{2}}(\sqrt{q}-1)}$, so ${2q^{k-1}(q+1)p^s< q^{\frac{n-1}{2}}(\sqrt{q}-1)}$ since ${p^s \leq q}$, and we get \linebreak[4]${2q^{k-1}(q+1)p^s+1\leq q^{\frac{n-1}{2}}(\sqrt{q}-1)}$ which ensures that  there exists a  place of $F_{k,s}/\F_{q^2}$ of degree $n$, according to Lemma~\ref{lemme_genre}~iii). On the other hand, we can proceed as the preceding proof to prove that for ${k\geq K(n)+1}$, with ${K(n):=\log_q\left(\frac{2n}{(q+1)(q-2)}\right)}$, Condition (2) is satisfied. To conclude, note that there exists at least one step $G_{k,s}/\F_{q}$ satisfying both Conditions (1) and (2) since for ${n\geq7}$ and ${q\geq7}$, ${\frac{n}{2}-K(n)-3\geq \frac{7}{2}- \log_7\Big(\frac{2\times 7}{8\times 5}\Big)-3>1}$. Moreover, remark that Condition (1) is satisfied from the step $G_{1,0}/\F_{q}$, so  the first step for which both Conditions (1) and (2) are verified is the first step for which (2) is verified.
\qed\\
\end{Proof}

In the special case where ${q=4}$, Condition (2) needs to be slightly stronger:

\begin{Lemma}\label{lemme_placedegnF4}
Let $n \geq 10$ be an integer. If ${q=p^2=4}$, then there exists a step $G_{k,s}/\F_4$ of the tower $T_3/\F_4$ such that the following conditions are verified:
\begin{enumerate}[(1)]
	\item there exists a place of $G_{k,s}/\F_4$ of degree $n$, 
	\item ${B_1(G_{k,s}/\F_4)+2B_2(G_{k,s}/\F_4) \geq 2n + g_{k,s}+2}$.
\end{enumerate}
Moreover, the first step for which both Conditions (1) and (2) are verified is the first step for which (2) is verified.
\end{Lemma}

\begin{Proof} We can proceed as in the previous proof with minor changes. Indeed, we first have that ${2g_{k,s}+1 \leq q^\frac{n-1}{2}(\sqrt{q}-1)}$ for ${1 \leq k \leq \frac{n-9/2}{2}}$ and ${s\in \{0,1\}}$, since in this case ${\frac{\sqrt{q}-1}{q+1}\cdot q^{\frac{n-1}{2}-k} = \frac{1}{5}2^{n-1-2k}\geq \frac{2^{7/2}}{5}>2}$, which proves that Condition (1) is verified according to Lemma~\ref{lemme_genre}~iii). Moreover, Condition (2) is
 satisfied for  ${k \geq K(n)+1}$ with ${K(n) := \log_4\left(\frac{2n+2}{(q+1)(q-2)}\right)}$, and  one can check that ${\frac{n}{2}-K(n)-\frac{9}{4} -1 \geq \frac{10}{2}-\frac{9}{4} - \log_4\left(\frac{20}{10}\right)>1}$.
\qed \\
\end{Proof}

This is a similar result for the tower $T_4/\F_2$:

\begin{Lemma}\label{lemme_placedegn_F2}
For any integer $n\geq12$ there exists a step $H_{k,s}/\F_2$ of the tower $T_4/\F_2$, with genus ${g_{k,s}\geq2}$, such that both following conditions are verified:
\begin{enumerate}[(1)]
	\item there exists a place of degree $n$ in $H_{k,s}/\F_2$, 
	\item $B_1(H_{k,s}/\F_2)+2B_2(H_{k,s}/\F_2)+4B_4(H_{k,s}/\F_2) \geq 2n + g_{k,s}+5$.
\end{enumerate}
Moreover, the first step for which both Conditions (1) and (2)  are verified is the first step for which (2) is verified. 
\end{Lemma}

\begin{Proof} According to \cite[Lemma 2.6]{bapi}, if $n\geq12$ then there exists a step $H_{k,s}/\F_2$ of the tower $T_4/\F_2$, with ${k\geq2}$ (so, in particular ${g_{k,s}\geq g_2=6}$) such that there exists a place of $H_{k,s}/\F_2$ of degree $n$ and \linebreak[4]${B_1(H_{k,s}/\F_2) + 2B_2(H_{k,s}/\F_2) +4B_4(H_{k,s}/\F_2) \geq 2n + 2g_{k,s} + 7}$. Thus we get the result since ${2n + 2g_{k,s} + 7 \geq 2n + g_{k,s}+5}$.
\qed \\
\end{Proof}

This is a similar result for the tower $T/\F_{p^2}$:

\begin{Lemma}\label{lemme_placedegngsr}
Let $p\geq3$ and $n \geq \frac{1}{2}\left(p^2+1+\epsilon(p^2)\right)$. There exists a step $L_k/\F_{p^2}$ of the tower $T/\F_{p^2}$, with genus ${g_k \geq 2}$, such that the  following conditions are verified:
\begin{enumerate}[(1)]
	\item there exists a place of $L_k/\F_{p^2}$ of degree $n$,
	\item $B_1(L_k/\F_{p^2}) \geq 2n + g_k - 1$.
\end{enumerate}
Moreover the first step for which both Conditions (1) and (2)  are verified is the first step for which (2) is verified.
\end{Lemma}

\begin{Proof} 
Note that ${n \geq \frac{1}{2}(3^2+1+2\cdot 3) = 8}$. We first prove that for all integers $k$ such that ${2 \leq k \leq n - 2}$, we have  ${2g_k+1 \leq p^{n-1}(p-1)}$ , so Condition (2) is satisfied. 
Indeed, for such an integer $k$, one has \linebreak[4]${2^{k+1} \leq 2^{n-1} < p^{n-1}}$,  since ${p>2}$. Thus  ${2\cdot 2^{k+1} < p^{n-1}(p-1)}$ since ${2\leq p-1}$ and we get the result from Lemma \ref{lemme_genregsr} ii).\\
We prove now that for ${k\geq \log_2 \big(\frac{n}{2}\big)}$, Condition (2) is verified. \mbox{Indeed,} for such an integer $k$, we have ${2^{k+2} \geq 2n}$, so ${2^{k+2}\geq 2n-2\cdot 2^\frac{k+1}{2}}$. Hence we get ${2^{k+1}(p-2) \geq 2n-2\cdot 2^\frac{k+1}{2}}$ since ${p\geq3}$ and then we obtain \linebreak[4]${2^{k+1}(p-1)\geq 2n + 2^{k+1}-2\cdot 2^\frac{k+1}{2}}$. Thus we have ${B_1(L_k/\F_{p^2}) \geq 2n + g_k - 1}$ according to Bound (\ref{nbratplgsr}) and Lemma~\ref{lemme_genregsr}~i).\\
Hence, we have proved that for any integers ${n\geq 8}$ and ${k \geq 2}$ such that\linebreak[4] ${\log_2\big(\frac{n}{2}\big) \leq k \leq  n - 2}$, both Conditions (1) and (2) are verified. Moreover, note that for any ${n\geq 8}$, there exists an integer $k \geq 2$ in the interval \linebreak[4]${\big[ \log_2 \big(\frac{n}{2}\big); n - 2 \big]}$ since ${n-2 - \log_2\big(\frac{n}{2}\big) \geq 6 - \log_2(4) >1}$. To conclude,  remark that Condition (1) is satisfied from the step ${L_{0}/\F_{p^2}}$, so  the first step for which both Conditions (1) and (2) are verified is the first step for which (2) is verified; moreover, for ${k\geq2}$, ${g_k \geq g_2=3}$.
\qed\\
\end{Proof}

This is a similar result for the tower $T/\F_p$:

\begin{Lemma}\label{lemme_placedegngsr2}
Let $p>5$ and $n \geq \frac{1}{2}\left(p+1+\epsilon(p)\right)$. There exists a step $L_k/\F_p$ of the tower $T/\F_p$, with genus ${g_k\geq2}$,  such that  the following conditions are verified:
\begin{enumerate}[(1)]
	\item there exists a place of $L_k/\F_p$ of degree $n$,
	\item $B_1(L_k/\F_p) + 2B_2(L_k/\F_p) \geq 2n + g_k - 1$.
\end{enumerate}
Moreover the first step for which both Conditions (1) and (2)  are verified is the first step for which (2) is verified.
\end{Lemma}

\begin{Proof} 
Note that ${n \geq \frac{1}{2}(7+1+\epsilon(7)) =7}$. 
We first prove that for all integers $k$ such that ${2 \leq k \leq n - 3}$, we have  ${2g_k+1 \leq p^\frac{n-1}{2}(\sqrt{p}-1)}$, so Condition (1) is satisfied.
Indeed, for such an integer $k$, one has \linebreak[4]${2^{k+2}\leq 2^{n-1}=4^\frac{n-1}{2}}$, so ${2 \cdot 2^{k+1}< p^\frac{n-1}{2}}$ since ${p>4}$.
Hence we get  \linebreak[4]${2\cdot 2^{k+1} < p^\frac{n-1}{2}(\sqrt{p}-1)}$, which gives the result from Lemma \ref{lemme_genregsr} ii).\\
On the other hand, we proceed as the preceding proof to prove that for ${k\geq \log_2 \big(\frac{n}{2}\big)}$, Condition (2) is verified. Moreover, note that for any ${n\geq 7}$, there exists an integer $k \geq 2$ in the interval ${\big[ \log_2 \big(\frac{n}{2}\big); n - 3 \big]}$ since \linebreak[4]${n-3-\log_2 \big(\frac{n}{2}\big) \geq 4 - \log_2 (3.5) >1}$. To conclude,  remark that Condition (1) is satisfied from the step ${L_{0}/\F_{p}}$, so  the first step for which both Conditions (1) and (2) are verified is the first step for which (2) is verified; moreover, for ${k\geq2}$, ${g_k \geq g_2=3}$.
\qed\\
\end{Proof}

This is a similar result for the tower $T/\F_p$ for $p=3$ or $5$:

\begin{Lemma}\label{lemme_placedegngsr3}
If ${p=5}$ and ${n\geq \frac{1}{2}(5+1+\epsilon(5))=5}$ or ${p=3}$ and ${n\geq11}$, then there exists a step $L_k/\F_p$ of the tower $T/\F_p$, with genus ${g_k\geq2}$,  such that  the following conditions are verified:
\begin{enumerate}[(1)]
	\item there exists a place of $L_k/\F_p$ of degree $n$,
	\item $B_1(L_k/\F_p) + 2B_2(L_k/\F_p) \geq 2n + g_{k,s} + 2$.
\end{enumerate}
Moreover the first step for which both Conditions (1) and (2)  are verified is the first step for which (2) is verified.
\end{Lemma}

\begin{Proof} 
We first consider  the case ${p=5}$ and ${n\geq5}$. Since ${p>4}$, the first part of the preceding proof shows that for all integers $k$ such that ${2 \leq k \leq n - 3}$, we have  ${2g_k+1 \leq p^\frac{n-1}{2}(\sqrt{p}-1)}$, so Condition (1) is satisfied. Now, we prove that for ${k\geq  \log_2 \big(\frac{n}{3}\big)}$, Condition (2) is satisfied. Indeed for such an integer $k$, one has ${2^{k+1}(p-2)+2^\frac{k+3}{2}\geq 2n+2\sqrt{\frac{2n}{3}}> 2n+3}$ since ${n\geq5}$. Thus we get ${2^{k+1}(p-1)> 2n+ (2^{k+1}-2^\frac{k+3}{2}+ 1)+2}$, which gives the result  according to Bound (\ref{nbratplgsr}) and Lemma~\ref{lemme_genregsr}~i). Hence, we have proved that for any integers ${n\geq 5}$ and ${k \geq 2}$ such that\linebreak[4] ${\log_2 \big(\frac{n}{3}\big) \leq k \leq  n - 3}$, both Conditions (1) and (2) are verified. Moreover, note that for any ${n\geq 5}$, there exists an integer $k \geq 2$ in the interval \linebreak[4]${\big[ \log_2 \big(\frac{n}{3}\big); n - 3 \big]}$ since ${n-3 - \log_2\big(\frac{n}{3}\big) \geq 2 - \log_2\big(\frac{n}{3}\big) >1}$. To conclude,  remark that Condition (1) is satisfied from the step ${L_{0}/\F_{p^2}}$, so  the first step for which both Conditions (1) and (2) are verified is the first step for which (2) is verified; moreover, for ${k\geq2}$, ${g_k \geq g_2=3}$.

Now we consider the case ${p=3}$ and ${n\geq11}$. We first prove that for all integers $k$ such that ${2 \leq k \leq \log_2(3^\frac{n-1}{2})-3}$, we have  ${2g_k+1 \leq 3^\frac{n-1}{2}(\sqrt{3}-1)}$, so Condition (1) is satisfied. Indeed, for such an integer $k$, one has \linebreak[4]${2^{k+3}\leq 3^\frac{n-1}{2}}$, so ${2 \cdot 2^{k+1}\leq \frac{1}{2}\cdot 3^\frac{n-1}{2}< 3^\frac{n-1}{2}(\sqrt{3}-1)}$ which gives the result from Lemma \ref{lemme_genregsr} ii). On the other hand, we prove that for ${k\geq  \log_2 (n)}$, Condition (2) is satisfied. Indeed for such an integer $k$, one has \linebreak[4]${2^{k+1}(p-2)+2^\frac{k+3}{2} = 2^{k+1}+2^\frac{k+3}{2}\geq 2n+2\sqrt{2n}> 2n+3}$ since ${n\geq11}$. Thus we get ${2^{k+1}(p-1)> 2n+ (2^{k+1}-2^\frac{k+3}{2}+ 1)+2}$, which gives the result  according to Bound (\ref{nbratplgsr}) and Lemma~\ref{lemme_genregsr}~i). Hence, we have proved that for any integers ${n\geq 11}$ and ${k \geq 2}$ such that ${\log_2 (n) \leq k \leq  \log_2(3^\frac{n-1}{2})-3}$, both Conditions (1) and (2) are verified. Moreover, note that for any ${n\geq 11}$, there exists an integer $k \geq 2$ in the interval ${\big[ \log_2 (n); \log_2(3^\frac{n-1}{2})-3 \big]}$ since ${\log_2(3^\frac{n-1}{2})-3 - \log_2(n)\geq \log_2(3^5)-3 - \log_2(11) >1}$. To conclude,  remark that Condition (1) is satisfied from the step ${L_{0}/\F_{p^2}}$, so  the first step for which both Conditions (1) and (2) are verified is the first step for which (2) is verified; moreover, for ${k\geq2}$, ${g_k \geq g_2=3}$.
\qed\\
\end{Proof}

\section{New uniform bounds for the tensor rank }

\begin{Theorem}
\label{theo_borneunifasym2}
For any integer $n\geq2$, we have
$$
\mu_{2}(n) \leq \frac{189}{22}n+18.
$$
\end{Theorem}

\begin{Proof}
Let ${q:=p^2=4}$ and ${n\geq 2}$. We apply the general method described in Section \ref{sectmethodegene} on the tower ${T_4/\F_q}$ with ${d=4}$,
${\gamma_{2,4} \leq \frac{3}{2}}$ (see Proof of Corollary \ref{chchq=2}) and ${\lambda := \frac{4\gamma_{2,4}}{\mus_2(4)}\leq \frac{2}{3}}$, since ${\mus_2(4)=9}$.\\
We set ${X= n_0^{k,s}+\frac{1}{2}(D_{k,s}-3)}$ where ${D_{k,s}=\frac{3}{2}p^{s+1}q^{k-1}}$. Lemmas \ref{lemme_delta_F2} and \ref{lemme_placedegn_F2} ensure that Hypotheses (\ref{hypgene1}) to (\ref{hypgene5}) are satisfied, so we have: 
\begin{eqnarray*}
\Phi(X) &  = & \frac{2\mus_q(d)}{d}\left(1+\frac{g(H_{k,s+1})}{2X}\right)X +\frac{\mus_q(d)}{d}(\alpha_q+d-1)\\
	& = & \frac{9}{2}\left(1+\frac{g(H_{k,s+1})}{2X}\right)X +18.
\end{eqnarray*}
From Lemmas  \ref{lemme_genre} iii) and \ref{lemme_bornesup_F2} it follows that:
\begin{eqnarray*}
\frac{g(H_{k,s+1})}{2X} & \leq & \frac{q^{k-1}(q+1)p^{s+1}}{2n_0^{k,s}+D_{k,s}-3}\\
		& \leq &  \frac{q^{k-1}(q+1)p^{s+1}}{5p^{s+1}q^{k-1}-5+\frac{3}{2}p^{s+1}q^{k-1}-3} \\
		& = & \frac{q+1}{\frac{13}{2}- \frac{8}{q^{k-1}p^{s+1}}}.
\end{eqnarray*}
Since $k\geq 2$, one has ${\frac{g(H_{k,s+1})}{2X} \leq \frac{10}{11}}$ which leads to ${\mu_q(n) \leq \frac{9}{2}\left(1+\frac{10}{11}\right)n +18}$ and gives the result.
\qed \\
\end{Proof}

\begin{Theorem}\label{theo_bornesunifasym}
Let $p$ be a prime and $q:=p^r$. For any $n\geq2$, we have:
\begin{enumerate}[(a)]
	\item\label{borneunifFq2} if $q\geq4$, then $${\mu_{q^2}(n) \leq 2\left(1+\frac{p}{q-2+(p-1)\frac{q}{q+1}}\right)n-1},$$
	\item\label{borneunifFp2} if $p\geq3$, then $$\mu_{p^2}(n) \leq 2\left(1+ \frac{2}{p-1}\right)n-1,$$
	\item\label{borneunifFq} if $q>5$, then $$\mu_{q}(n) \leq 3\left(1+\frac{p}{q-2+(p-1)\frac{q}{q+1}}\right)n,$$
	\item\label{borneunifFp} if $p>5$, then $$\mu_{p}(n) \leq 3\left(1+\frac{2}{p-1}\right)n.$$
\end{enumerate}
\end{Theorem}

\begin{Proof}
\begin{enumerate}[(a)]
	\item Let  ${n\geq \frac{1}{2}(q^2+1+\epsilon(q^2))}$. We apply the general method described in Section \ref{sectmethodegene} on the tower ${T_2/\F_{q^2}}$ with ${d=1}$,
${\gamma_{q^2,1} \leq 1}$ (see Proof of Corollary \ref{chchq>5}) and ${\lambda := \frac{\gamma_{q^2,1}}{\mus_{q^2}(1)} \leq 1}$.\\
We set ${X= n_0^{k,s}+\frac{1}{2}D_{k,s}}$ where ${D_{k,s}=(p-1)p^{s}q^{k}}$. Lemmas \ref{lemme_deltaFq} and \ref{lemme_placedegn_Fq2} ensure that Hypotheses (\ref{hypgene1}) to (\ref{hypgene5}) are satisfied. Note that we  can always choose a step $F_{k,s+1}$ with ${k\geq4}$ (so in particular ${g_{k,s+1}\geq2}$), even if  doing so we may have a non-optimal bound for some small $n$.\\
Thus we have:
\begin{eqnarray*}
\Phi(X) &  = & 2\left(1+\frac{g(F_{k,s+1})}{2X}\right)X -1
\end{eqnarray*}
From Lemmas  \ref{lemme_genre} iii) and \ref{lemme_bornesupFq}  it follows that:
\begin{eqnarray*}
\frac{g(F_{k,s+1})}{2X} & \leq & \frac{q^{k-1}(q+1)p^{s+1}}{2n_0^{k,s}+D_{k,s}}\\
		& \leq &  \frac{q^{k-1}(q+1)p^{s+1}}{(q+1)q^{k-1}p^{s}(q-2)+(p-1)p^{s}q^{k}} \\
		& = & \frac{p}{q-2+(p-1)\frac{q}{q+1}}
\end{eqnarray*}
which gives the result.\\
	\item Let  ${n\geq \frac{1}{2}(p^2+1+\epsilon(p^2))}$. We apply the general method described in Section \ref{sectmethodegene} on the tower ${T/\F_{p^2}}$ with ${d=1}$,
${\gamma_{p^2,1} \leq 1}$ and ${\lambda := \frac{\gamma_{p^2,1}}{\mu_{p^2}(1)} \leq 1}$.\\
We set ${X= n_0^{k}+\frac{1}{2}D_{k}}$ where ${D_{k}=2^{k+1}-2^\frac{k+1}{2}}$. Lemmas \ref{lemme_deltagsr} and \ref{lemme_placedegngsr} ensure that Hypotheses (\ref{hypgene1}) to (\ref{hypgene5}) are satisfied.\\
Thus we have:
\begin{eqnarray*}
\Phi(X) &  = & 2\left(1+\frac{g(L_{k+1})}{2X}\right)X -1
\end{eqnarray*}
From Lemmas  \ref{lemme_genregsr} ii) and \ref{lemme_bornesupgsr}  it follows that:
\begin{eqnarray*}
\frac{g(L_{k+1})}{2X} & \leq & \frac{2^{k+2}}{2n_0^{k,s}+D_{k,s}}\\
		& \leq &  \frac{2^{k+2}}{2^{k+1}(p-2)+2\frac{k+3}{2}+2^{k+1}-2^\frac{k+1}{2}} \\
		& = &\frac{2}{p-1+2^{-\frac{k-1}{2}} -2^{-\frac{k+1}{2}}}
\end{eqnarray*}
which gives the result, since ${2^{-\frac{k-1}{2}} -2^{-\frac{k+1}{2}} \geq0}$.
	\item Let  ${n\geq \frac{1}{2}(q+1+\epsilon(q))}$. We apply the general method described in Section \ref{sectmethodegene} on the tower ${T_3/\F_q}$ with ${d=2}$,
${\gamma_{q,2} \leq \frac{1}{2}}$  (see Proof of Corollary \ref{chchq>5}) and ${\lambda := \frac{2\gamma_{q,2}}{\mus_q(2)} \leq \frac{1}{3}}$ since ${\mus_q(2)\geq3}$.\\
We set ${X= n_0^{k,s}+\frac{1}{2}(D_{k,s}-1)}$ where ${D_{k,s}=(p-1)p^{s}q^{k}}$. Lemmas \ref{lemme_deltaFq} and \ref{lemme_placedegn_Fq} ensure that Hypotheses (\ref{hypgene1}) to (\ref{hypgene5}) are satisfied. Note that we  can always choose a step $F_{k,s+1}$ with ${k\geq4}$ (so in particular ${g_{k,s+1}\geq2}$), even if  doing so we may have a non-optimal bound for some small $n$.\\
Thus we have:
\begin{eqnarray*}
\Phi(X) &  = & 3\left(1+\frac{g(G_{k,s+1})}{2X}\right)X.
\end{eqnarray*}
We proceed as in (\ref{borneunifFq2}) to get ${\frac{g(G_{k,s+1})}{2X} \leq \frac{p}{q-2+(p-1)\frac{q}{q+1}}}$ which gives the result. (Note that ${\lambda \leq 1}$ so Lemma \ref{lemme_deltaFq} implies that Hypothesis (\ref{hypgene4}) of Section \ref{sectmethodegene} is satisfied.)\\
	\item Let  ${n\geq \frac{1}{2}(p+1+\epsilon(p ))}$. We apply the general method described in Section \ref{sectmethodegene} on the tower ${T/\F_p}$ with ${d=2}$, 
${\gamma_{p,2} \leq \frac{1}{2}}$ (see Proof of Corollary \ref{chchq>5}) and ${\lambda := \frac{2 \gamma_{p,2}}{3} \leq \frac{1}{3}}$. \\
We set  ${X= n_0^{k}+\frac{1}{2}(D_{k}-1)}$ where ${D_{k}=2^{k+1}-2^\frac{k+1}{2}}$.  Lemmas \ref{lemme_deltagsr} and \ref{lemme_placedegngsr2} ensure that Hypotheses (\ref{hypgene1}) to (\ref{hypgene5}) are satisfied.\\
Thus we have:
\begin{eqnarray*}
\Phi(X) &  = & 3\left(1+\frac{g(L_{k+1})}{2X}\right)X
\end{eqnarray*}
We proceed as in (\ref{borneunifFp2}) to get ${\frac{g(L_{k+1})}{2X} \leq  \frac{2}{p-1}}$ which gives the result. (Note that ${\lambda \leq 1}$ so Lemma \ref{lemme_deltagsr} implies that Hypothesis (\ref{hypgene4}) of Section \ref{sectmethodegene} is satisfied.)
\qed\\
\end{enumerate}
\end{Proof}

\begin{Theorem}\label{theo_bornesunifasym35}
For any  $n\geq2$, we have
$$
\mu_{3}(n) \leq 6n,
\qquad
\mu_4(n) \leq \frac{87}{19}n,
\qquad \mbox{
and
} \qquad
\mu_{5}(n) \leq \frac{9}{2}n.
$$
\end{Theorem}

\begin{Proof} 
For the bounds over $\F_3$ and $\F_5$, we proceed as in the proof of Theorem \ref{theo_bornesunifasym} (\ref{borneunifFp}), since Lemma \ref{lemme_placedegngsr3} ensures that the method is still valid in this cases. Thus we get 
$$
\mu_p(n)\leq 3\left(1+\frac{2}{p-1}\right).
$$
Note that with our method,  we prove the bound for ${\mu_3(n)}$ for ${n\geq11}$ according to Lemma \ref{lemme_placedegngsr3}, but that  this bound holds also for ${n\leq10}$, according to Table 1 in \cite{ceoz}. 

The bound over $\F_4$ is obtained for ${n\geq10}$ with the same reasoning as in the proof of Theorem \ref{theo_bornesunifasym} (\ref{borneunifFq}): let  ${q := 4}$ and ${n\geq 10 > \frac{1}{2}(q+1+\epsilon(q))}$, we apply the general method described in Section \ref{sectmethodegene} on the tower ${T_3/\F_4}$ with ${d=2}$,
${\gamma_{4,2} \leq 1}$  (see Proof of Corollary \ref{chchq>5}) and ${\lambda := \frac{2\gamma_{4,2}}{\mus_4(2)} \leq \frac{2}{3}}$ since ${\mus_4(2)\geq3}$. We set ${X= n_0^{k,s}+\frac{1}{2}(D_{k,s}-1)}$ where ${D_{k,s}=(p-1)p^{s}q^{k-1}}$. Lemmas \ref{lemme_deltaFq} and \ref{lemme_placedegnF4} ensure that Hypotheses (\ref{hypgene1}) to (\ref{hypgene5}) are satisfied. Note that we  can always choose a step $F_{k,s+1}$ with ${k\geq4}$ (so in particular\linebreak[4]${g_{k,s+1}\geq2}$), even if  doing so we may have a non-optimal bound for some small $n$. Thus we have:
\begin{eqnarray*}
\Phi(X) &  = & 3\left(1+\frac{g(G_{k,s+1})}{2X}\right)X
\end{eqnarray*}
which gives ${\frac{g(G_{k,s+1})}{2X} \leq \frac{p}{q-2+(p-1)\frac{q}{q+1}}}$. (Note that ${\lambda \leq 1}$ so Lemma \ref{lemme_deltaFq} implies that Hypothesis (\ref{hypgene4}) of Section \ref{sectmethodegene} is satisfied.) To conclude, remark that our bound is still valid for $\mu_4(n)$ when  ${4.5 = \frac{1}{2}(q+1+\epsilon(q))\leq n <10}$ according to the known estimates for $\mus_4(n)$  (recalled in \cite[Table 1]{ceoz}).
\qed\\
\end{Proof}

\section{Asymptotic bounds}
\label{section_asympt}

So far we gave upper bounds for the tensor rank of multiplication
that hold uniformly for any extension of finite fields.
Now, introducing the quantity
$$
M_q=\limsup_{n\to\infty}\frac{\mu_q(n)}{n}
$$
and letting the degree
of the extension go to infinity,
these bounds then turn into the following asymptotic
estimates:

\begin{Proposition}
\label{prop_asympt}
We have
$$M_2 \leq \frac{189}{22}\approx 8.591, \qquad M_3\leq6, \qquad M_4 \leq \frac{87}{19}\approx 4.579, \qquad M_5\leq 4.5,$$
and for $p$ a prime and $q=p^r$,
\begin{enumerate}[(a)]
	\item if $q\geq4$, then $M_{q^2}\leq 2\left(1+\frac{p}{q-2+(p-1)\frac{q}{q+1}}\right),$
	\item if $p\geq3$, then $M_{p^2}\leq 2\left(1+ \frac{2}{p-1}\right),$
	\item if $q>5$, then $M_{q}\leq 3\left(1+\frac{p}{q-2+(p-1)\frac{q}{q+1}}\right),$
	\item if $p>5$, then $M_{p}\leq 3\left(1+\frac{2}{p-1}\right).$
\end{enumerate}
\end{Proposition}
\begin{Proof}
Let $n\to\infty$ in Theorems \ref{theo_borneunifasym2}, \ref{theo_bornesunifasym}, and \ref{theo_bornesunifasym35}.\qed\\
\end{Proof}

It is interesting to compare these asymptotic bounds with other known similar
results, such as the ones in \cite{cacrxiya}. We see the bound on $M_2$ in
Proposition \ref{prop_asympt} is less sharp than the one in \cite{cacrxiya},
while the bounds on $M_3$, $M_4$, and $M_5$ are better.

However, in such a comparison, one should keep in mind other features
of these various bounds. On one hand, the bounds in \cite{cacrxiya} hold not only
for the general bilinear complexity, but also for the symmetric
bilinear complexity. On the other hand, the constructions leading
to Proposition \ref{prop_asympt} were not aimed solely at maximizing
asymptotics:
\begin{itemize}
\item they give uniform bounds, that hold for any given extension of
  finite fields (so, not only asymptotically)
\item they come from towers of curves given by explicit equations, so
 at least in principle, it should be possible to write explicitly the
 multiplication algorithms reaching these bounds.
\end{itemize}

Now, if one relaxes these last two conditions, it is possible to give
substantially better asymptotic bounds, especially for $q$ small.
For this we will borrow the following lemma from \cite{cacrxiya}
(with a very slight modification):

\begin{Lemma}[compare \cite{cacrxiya}, Lemma IV.4]
\label{ShimuraElkies}
Let $q$ be a prime power and $t\geq1$ an integer such
that \emph{$q^t$ is a square} (so $q$ itself is a square, or $t$ is even).
Then there exists a family $(F_s/\F_q)_{s\geq1}$ of function fields
such that, as $s$ goes to infinity, we have:
\begin{enumerate}[(i)]
\item $g_s\to\infty$
\item $g_{s+1}/g_s\to1$
\item $B_{t}(F_s)/g_s\to(q^{t/2}-1)/t$
\end{enumerate}
where $g_s$ is the genus of $F_s/\F_q$.
\end{Lemma}

For the details of the proof we refer to \cite{cacrxiya},
where it is in fact credited to Elkies, who proceeded by modifying
the construction of Shimura curves previously introduced in \cite{shtsvl}.

As a matter of fact, the version of the lemma originally stated
in \cite{cacrxiya} requires $t$ even, while we allow $t$ odd provided $q$ is
a square. However our increased generality is only apparent, because
it is readily seen that the aforementioned proof of Elkies also gives
the version we stated. Alternatively, when $q$ is a square, we can replace
$q$ and $t$ with $q^{1/2}$ and $2t$ to reduce to the case $t$ even, and
conclude with a base field extension argument.

\begin{Theorem}
\label{th_asymp}
Let $q$ be a prime power and $t\geq1$ an integer such
that $q^t\geq9$ is a square.
Then
$$
M_q\leq\frac{2\mu_q(t)}{t}\left(1+\frac{1}{q^{t/2}-2}\right).
$$
\end{Theorem}
\begin{proof}
Let $(F_s/\F_q)_{s\geq1}$ be the family of function fields given
by Lemma~\ref{ShimuraElkies} for $q$ and $t$.
Given an integer $n$, let $s(n)$ be the smallest integer such that
$$
tB_t(F_{s(n)}/\F_q)-g_{s(n)}\geq 2n+8.
$$
Such an integer exists
because of conditions (i) and (iii) in
Lemma~\ref{ShimuraElkies} and our hypothesis $q^t\geq9$,
and it goes to infinity with $n$.
More precisely, minimality of $s(n)$ and conditions (iii) and (ii) give,
respectively:
\begin{itemize}
\item $tB_t(F_{s(n)}/\F_q)-g_{s(n)}\geq 2n+8>tB_t(F_{s(n)-1}/\F_q)-g_{s(n)-1}$
\item $tB_t(F_{s(n)}/\F_q)=(q^{t/2}-1)g_{s(n)}+o(g_{s(n)})$
\item $g_{s(n)-1}=g_{s(n)}+o(g_{s(n)})$
\end{itemize}
hence the estimate
$$
(q^{t/2}-2)g_{s(n)}+o(g_{s(n)})=2n+o(n)
$$
which can be restated finally as
$$
g_{s(n)}=\frac{2n}{q^{t/2}-2}+o(n)
$$
and
$$
B_t(F_{s(n)}/\F_q)=\frac{2n}{t}\left(1+\frac{1}{q^{t/2}-2}\right)+o(n).
$$
The estimate on $g_{s(n)}$ implies 
$2g_{s(n)}+1\leq q^{(n-1)/2}(q^{1/2}-1)$ as soon as $n$ is big enough.
We can then use Theorem~\ref{theobornegenerale} with $F_{s(n)}/\F_q$,
setting $m=n$, $l=1$, $N_t=n_{t,1}=B_t(F_{s(n)}/\F_q)$,
and $n_{d,u}=0$ for all other values of $d$ and $u$.
This gives
$$
\mu_q(n)\leq\mu_q(t)B_t(F_{s(n)}/\F_q)
$$
and the conclusion follows.
\end{proof}

\begin{Corollary}
We have:
$$
M_2\leq 35/6\approx 5.833
$$
$$
M_3\leq 36/7\approx 5.143
$$
$$
M_4\leq 30/7\approx 4.286
$$
\end{Corollary}
\begin{proof}
Apply Theorem \ref{th_asymp}
with $q=2$, $t=6$, $\mu_2(6)\leq 15$;
with $q=3$, $t=4$, $\mu_3(4)\leq 9$;
and with $q=4$, $t=4$, $\mu_4(4)\leq 8$.
\end{proof}

\begin{Corollary}
For any $q\geq3$ we have $M_q\leq 3\left(1+\frac{1}{q-2}\right)$.
In particular:
$$
M_5\leq 4
$$
$$
M_7\leq 3.6
$$
$$
M_8\leq 3.5
$$
\end{Corollary}
\begin{proof}
Apply Theorem \ref{th_asymp} with $t=2$, $\mu_q(2)=3$.
\end{proof}


\begin{thebibliography}{10}

\bibitem{ball1}
St\'ephane Ballet.
\newblock Curves with many points and multiplication complexity in any
  extension of ${\F}_q$.
\newblock {\em {F}inite {F}ields and {T}heir {A}pplications}, 5:364--377, 1999.

\bibitem{ball3}
St\'ephane Ballet.
\newblock Low increasing tower of algebraic function fields and bilinear
  complexity of multiplication in any extension of ${\F}_q$.
\newblock {\em {F}inite {F}ields and {T}heir {A}pplications}, 9:472--478, 2003.

\bibitem{ball5}
St\'ephane Ballet.
\newblock On the tensor rank of the multiplication in the finite fields.
\newblock {\em {J}ournal of {N}umber {T}heory}, 128:1795--1806, 2008.

\bibitem{bach}
St\'ephane Ballet and Jean Chaumine.
\newblock On the bounds of the bilinear complexity of multiplication in some
  finite fields.
\newblock {\em {A}pplicable {A}lgebra in {E}ngineering {C}ommunication and
  {C}omputing}, 15:205--211, 2004.

\bibitem{balb}
St\'ephane Ballet and Dominique Le~Brigand.
\newblock On the existence of non-special divisors of degree $g$ and ${g-1}$ in
  algebraic function fields over ${\F}_q$.
\newblock {\em {J}ournal on {N}umber {T}heory}, 116:293--310, 2006.

\bibitem{balbro}
St\'ephane Ballet, Dominique Le~Brigand, and Robert Rolland.
\newblock On an application of the definition field descent of a tower of
  function fields.
\newblock In {\em {P}roceedings of the {C}onference {A}rithmetic, {G}eometry
  and {C}oding {T}heory ({AGCT} 2005)}, volume~21, pages 187--203.
  {S}oci\'et\'e {M}ath\'ematique de {F}rance, s\'er. {S}\'eminaires et
  {C}ongr\`es, 2009.

\bibitem{bapi}
St\'ephane Ballet and Julia Pieltant.
\newblock On the tensor rank of multiplication in any extension of ${\F}_2$.
\newblock {\em {J}ournal of {C}omplexity}, 27:230--245, 2011.

\bibitem{baro1}
St\'ephane Ballet and Robert Rolland.
\newblock Multiplication algorithm in a finite field and tensor rank of the
  multiplication.
\newblock {\em {J}ournal of {A}lgebra}, 272(1):173--185, 2004.

\bibitem{baro4}
St\'ephane Ballet and Robert Rolland.
\newblock Families of curves over any finite field attaining the generalized
  {D}rinfeld-{V}ladut bound.
\newblock {\em Publ. Math. Univ. Franche-Comt\'e Besan\c con Alg\`ebr. Theor.
  Nr.}, pages 5--18, 2011.

\bibitem{bash}
Ulrich Baum and Amin Shokrollahi.
\newblock An optimal algorithm for multiplication in ${\F}_{256}/{\F}_4$.
\newblock {\em Applicable Algebra in Engineering, Communication and Computing},
  2(1):15--20, 1991.

\bibitem{cacrxiya}
Ignacio Cascudo, Ronald Cramer, Chaoping Xing, and An~Yang.
\newblock Asymptotic bound for multiplication complexity in the extensions of
  small finite fields.
\newblock {\em {IEEE} {T}ransactions on {I}nformation {T}heory},
  58(7):4930--4935, 2012.

\bibitem{ceoz}
Murat Cenk and Ferruh {\"O}zbudak.
\newblock On multiplication in finite fields.
\newblock {\em {J}ournal of {C}omplexity}, 26(2):172--186, 2010.

\bibitem{chch}
David Chudnovsky and Gregory Chudnovsky.
\newblock Algebraic complexities and algebraic curves over finite fields.
\newblock {\em {J}ournal of {C}omplexity}, 4:285--316, 1988.

\bibitem{groo}
Hans~F. de~Groote.
\newblock Characterization of division algebras of minimal rank and the
  structure of their algorithm varieties.
\newblock {\em {S}{I}{A}{M} {J}ournal on {C}omputing}, 12(1):101--117, 1983.

\bibitem{gast}
Arnaldo Garcia and Henning Stichtenoth.
\newblock A tower of {A}rtin-{S}chreier extensions of function fields attaining
  the {D}rinfeld-{V}ladut bound.
\newblock {\em {I}nventiones {M}athematicae}, 121:211--222, 1995.

\bibitem{gast2}
Arnaldo Garcia, Henning Stichtenoth, and Hans-Georg R{\"u}ck.
\newblock On tame towers over finite fields.
\newblock {\em {J}ournal fur die reine und angewandte {M}athematik},
  557:53--80, 2003.

\bibitem{rand3}
Hugues Randriambololona.
\newblock Bilinear complexity of algebras and the {C}hudnovsky-{C}hudnovsky
  interpolation method.
\newblock {\em {J}ournal of {C}omplexity}, 28:489--517, 2012.

\bibitem{shok}
Amin Shokrollahi.
\newblock Optimal algorithms for multiplication in certain finite fields using
  algebraic curves.
\newblock {\em {S}{I}{A}{M} {J}ournal on {C}omputing}, 21(6):1193--1198, 1992.

\bibitem{shtsvl}
Igor Shparlinski, Michael Tsfasman, and Serguei Vl\u{a}du\c{t}.
\newblock Curves with many points and multiplication in finite fields.
\newblock In {H}. {S}tichtenoth and {M}.{A}. {T}sfasman, editors, {\em {C}oding
  {T}heory and {A}lgebraic {G}eometry}, number 1518 in {L}ectures {N}otes in
  {M}athematics, pages 145--169, Berlin, 1992. Springer-Verlag.
\newblock Proceedings of {AGCT}-3 {C}onference, June 17-21, 1991, {L}uminy.

\bibitem{wino3}
Shmuel Winograd.
\newblock On multiplication in algebraic extension fields.
\newblock {\em {T}heoretical {C}omputer {S}cience}, 8:359--377, 1979.

\end{thebibliography}
\end{document}